\documentclass[10pt, oneside]{article} 
\usepackage{amsmath,amssymb}  
\usepackage{mathrsfs,amsthm}
\usepackage{graphicx}
\usepackage{stmaryrd}
\usepackage[all]{xy}
\usepackage{amscd}
\usepackage{geometry}
\usepackage{empheq}
\makeatletter
\@addtoreset{equation}{section}

\makeatother

\title{On log local Cartier transform of higher level in characteristic $p$}
\author{Sachio Ohkawa} 
\date{}

\begin{document}

\maketitle
\begin{abstract}
In our previous paper, given an integral log smooth morphism $X\to S$ of fine log schemes of characteristic $p>0$, we studied the Azumaya nature of the sheaf of log differential operators of higher level and constructed a splitting module of it under the existence of a certain lifting modulo $p^{2}$.
In this paper, under a certain liftability assumption which is stronger than our previous paper, we construct another splitting module of our Azumaya algebra over a  scalar extension which is smaller than our previous paper. 
As an application, we construct an equivalence, which we call the log local Cartier transform of higher level, between certain ${\cal D}$-modules and certain Higgs modules.
We also discuss the compatibility of the log Frobenius descent and the log local Cartier transform and the relation between the splitting module constructed in this paper and that constructed in the previous paper.
Our result can be considered as a generalization of the result of Ogus-Vologodsky, Gros-Le Stum-Quir\'os to the case of log schemes and that of Schepler
to the case of higher level.
\end{abstract}
\theoremstyle{plain}
\newtheorem{theo}{Theorem}[section]
\newtheorem{defi}[theo]{Definition}
\newtheorem{lemm}[theo]{Lemma}
\newtheorem{pro}[theo]{Proposition}
\newtheorem{cor}[theo]{Corollary}
\newtheorem{Prod}[theo]{Proposition-Definition}

\theoremstyle{definition}
\newtheorem{rem}[theo]{Remark}
\newtheorem{ex}[theo]{Example}

\renewcommand*\proofname{\upshape{\bfseries{Proof}}}

\section{Introduction}
Let $X\to S$ be a smooth scheme in positive characteristic.
Ogus-Vologodsky \cite{OV} established two Simpson type equivalences of categories between certain $\cal D$-modules on $X$ and certain Higgs modules on $X'$, thus establishing a nonabelian Hodge theory in characteristic $p$.
Here $X'\to S$ denotes the Frobenius pull back of $X\to S$.
Their construction is based on the Azumaya nature of the sheaf ${\cal D}_{X/S}^{(0)}$ of differential operators of level $0$ over its center.
Actually, ${\cal D}_{X/S}^{(0)}$ contains the symmetric algebra $S^{\cdot}{\cal T}_{X'/S}$ of the tangent bundle of $X'\to S$ as the center via the $p$-curvature map.
Ogus-Vologodsky constructed a splitting module of ${\cal D}_{X/S}^{(0)}$ over $\hat{\Gamma}_{\cdot}{\cal T}_{X'/S}$ under the existence of a lifting of $X'\to S$ modulo $p^{2}$.
Here $\hat{\Gamma}_{\cdot}{\cal T}_{X'/S}$ denotes the completion of the PD algebra ${\Gamma}_{\cdot}{\cal T}_{X'/S}$ by its augmentation ideal.
They also constructed a splitting module of ${\cal D}_{X/S}^{(0)}$ over $\hat{S}^{\cdot}{\cal T}_{X'/S}$ under the existence of a lifting of the relative Frobenius modulo $p^{2}$.
Here $\hat{S}^{\cdot}{\cal T}_{X'/S}$ denotes the completion of $S^{\cdot}{\cal T}_{X'/S}$ by its augmentation ideal.
The latter assumption is obviously stronger than the former assumption.
Furthermore, it is known that a lifting of the relative Frobenius modulo $p^{2}$ rarely exists globally on $X$.
In this sense, they call the equivalence obtained from the former splitting module the global Cartier transform and that from the latter splitting module the local Cartier transform.
Various generalizations of the Cartier transform such as the logarithmic version \cite{S} and the case of the sheaf ${\cal D}_{X/S}^{(m)}$ of differential operators of level $m$ \cite{GLQ} are studied by Schepler and Gros-Le Stum-Quir\'os.
Let us first recall the technical construction of the higher level version of the local Cartier transform \cite{GLQ}.

\subsection{The local Cartier transform of higher level}
Let $X\to S$ be a smooth morphism of schemes in positive characteristic of the relative dimension $r$.
We consider the following commutative diagram:
\[\xymatrix{
{X}\ar[r] \ar@/^1pc/[rr]^{F_{X}}\ar@/_/[rd] &\ar@{}[rd]|{\square} {X'} \ar[d] \ar[r] & X \ar[d] \\
& S \ar[r]^{F_{S}} & S ,}
\]
where $F_{X}$ (resp. $F_{S}$) denotes the $(m+1)$-st composition of the absolute Frobenius of $X$ (resp. $S$)
and the right square is cartesian.
We denote $X\to X'$ by $F_{X/S}$ and call it the $(m+1)$-st relative Frobenius of $X\to S$.
Let ${\cal P}_{X/S, (m)}$ be the structure sheaf of the $m$-PD envelope of the diagonal $X\to X\times_{S} X$ and ${\Gamma}_{\cdot}{\Omega}^{1}_{X'/S}$ the PD algebra defined by the cotangent bundle ${\Omega}^{1}_{X'/S}$ of $X'\to S$.
Assume that we are given a strong lifting which is a certain lifting of $F_{X/S}$ modulo $p^{2}$ (for the definition of strong lifting, see [GLQ, Definition 4.1]).
Gros-Le Stum-Quir\'os constructed the morphism of PD-algebras [GLQ, Proposition 4.7]
\begin{equation}\label{e7}
\Psi: F^{*}_{X/S}{\Gamma}_{\cdot}{\Omega}^{1}_{X'/S}\to {\cal P}_{X/S, (m)}
\end{equation}
and proved that the ${\cal O}_{X}$-dual of the scalar extension $F_{X/S}^{*}{\Gamma}_{\cdot}{\Omega}^{1}_{X'/S}\otimes_{{\cal O}_{X}} {{\cal O}_{X\times_{X'}X}}\to {\cal P}_{X/S, (m)}$ of $\Psi$ defines a splitting isomorphism [GLQ, Theorem 4.13] 
\begin{equation}\label{e5}
\hat{S^{\cdot}}{\cal T}_{X'/S}\otimes_{S^{.}{\cal T}_{X'/S}}{\cal D}_{X/S}^{(m)}\xrightarrow{\cong}{\cal E}nd_{\hat{S^{\cdot}}{\cal T}_{X'/S}}(F_{X/S}^{*}\hat{S^{\cdot}}{\cal T}_{X'/S}).
\end{equation}
Here we regard ${\cal D}_{X/S}^{(m)}$ as a $S^{\cdot}{\cal T}_{X'/S}$-module via the higher level version of the classical
$p$-curvature map constructed in the section 3 of \cite{GLQ}.
Note that, because $F_{X/S}$ is finite flat of rank $p^{r(m+1)}$, $F_{X/S}^{*}\hat{S^{\cdot}}{\cal T}_{X'/S}$ is a locally free $\hat{S^{\cdot}}{\cal T}_{X'/S}$-module of rank $p^{r(m+1)}$.
Then, by the Morita equivalence, they obtained an equivalence of categories [GLQ, Proposition 5.7]
\begin{equation}\label{e1}
\left(
		\begin{array}{c}
		\text{The category of} \\
		\text{left $\hat{S^{\cdot}}{\cal T}_{X'/S}\otimes_{S^{.}{\cal T}_{X'/S}}{\cal D}_{X/S}^{(m)}$-modules on $X$}
		\end{array}
	\right)
	\xrightarrow{\cong} 
	\left(
		\begin{array}{c}
		\text{The category of} \\
		\text{$\hat{S^{\cdot}}{\cal T}_{X'/S}$-modules on $X'$}
		\end{array}
	\right).
\end{equation}
This is regarded as the higher level version of [OV, Theorem 2.11]. 
We call it the local Cartier transform of higher level.
Note that, in the proof of (\ref{e5}), the local freeness of $F_{X/S}$ is essential for their dual calculation.
They also proved the Azumaya nature of ${\cal D}_{X/S}^{(m)}$ over $S^{\cdot}{\cal T}_{X'/S}$ [GLQ, Proposition 3.6],
but they does not use it for the proof of (\ref{e5}).

\subsection{Logarithmic version} \label{ss1}
One of aims of this paper is to construct the logarithmic version of the local Cartier transform of higher level.
Let $X\to S$ be an integral log smooth morphism of fine log schemes in positive characteristic.
A difficulty in the logarithmic case is the fact that the relative Frobenius of $X\to S$ is not necessarily finite flat even when $X\to S$ is log smooth.
Sometimes we cannot simply generalize the theory of differential modules in positive characteristic to the case of log schemes.
For example, it was well-known that the Cartier descent cannot be generalized in a naive manner.
We will overcome this difficulty by using Lorenzon-Montagnon indexed algebras ${\cal A}_{X}^{gp}$ and ${\cal B}^{(m+1)}_{X/S}$ associated to log structures and 
modify the proof due to Gros-Le Stum-Quir\'os to work in the logarithmic case.
Roughly speaking, ${\cal A}^{gp}_{X}$ and ${\cal B}^{(m+1)}_{X/S}$ are algebras indexed by ${\cal M}^{gp}_{X}/{\cal O}^{*}_{X}$ and there is a natural inclusion ${\cal B}^{(m+1)}_{X/S} \hookrightarrow {\cal A}^{gp}_{X}$
which is locally free in the category of ${\cal M}^{gp}_{X}/{\cal O}^{*}_{X}$-indexed modules, which gives a natural extension of $F^{*}_{X/S}: {\cal O}_{X'}\to {\cal O}_{X}$.
Here ${\cal M}_{X}^{gp}$ denotes the group envelope of the log structure of $X$ and ${\cal O}^{*}_{X}$ denotes the sheaf of invertible functions on $X$.

Let $F_{X/S}$ be the $(m+1)$-st relative Frobenius $X\to X'$ of $X\to S$.
Let ${\cal P}_{X/S, (m)}$ be the structure sheaf of the log $m$-PD envelope of the diagonal $X\to X\times_{S} X$.
In our previous paper \cite{O}, we proved the Azumaya nature of the indexed version of the sheaf
\begin{equation*}
\tilde{{\cal D}}^{(m)}_{X/S}:={\cal A}^{gp}_{X}\otimes_{{\cal O}_{X}}{\cal D}^{(m)}_{X/S}
\end{equation*}
of logarithmic differential operators of level $m$ over its center, which is isomorphic to ${\cal B}^{(m+1)}_{X/S}\otimes_{{\cal O}_{X'}} S^{.}{\cal T}_{X'/S}$
via the $p^{m+1}$-curvature map.
Here ${\cal T}_{X'/S}$ denotes the log tangent bundle of $X'\to S$.

In this paper, we will first define the notion of log strong lifting (Definition \ref{d1}), which is a certain lifting of the $(m+1)$-st relative Frobenius of $X\to S$ modulo $p^{2}$ and generalizes the notion of strong lifting defined  in [GLQ, Definition 4.4] to the case of log schemes.
Under the existence of a log strong lifting, we construct the divided Frobenius map (\ref{d6})
\begin{equation*}
\Psi: F^{*}_{X/S}{\Gamma}_{\cdot}{\Omega}^{1}_{X'/S}\to {\cal P}_{X/S, (m)},
\end{equation*}
which generalizes (\ref{e7}) to the case of log schemes.
We next prove the following theorem (Theorem \ref{t8}), which is the log version of (\ref{e5}).
\begin{theo}\label{t8}
Let $X\to S$ be an integral log smooth morphism in positive characteristic with a log strong lifting.
Then there exists an isomorphism of ${\cal M}^{gp}_{X}/{\cal O}^{*}_{X}$-indexed algebras
\begin{equation*}
{\cal B}_{X/S}^{(m+1)}\otimes_{{\cal O}_{X'}}\hat{S^{\cdot}}{\cal T}_{X'/S} \otimes_{{\cal B}_{X/S}^{(m+1)}\otimes_{{\cal O}_{X'}}S^{\cdot}{\cal T}_{X'/S}}\tilde{{\cal D}}^{(m)}_{X/S}\xrightarrow{\cong}
{\cal E}nd_{{\cal B}_{X/S}^{(m+1)}\otimes_{{\cal O}_{X'}}\hat{S^{\cdot}}{\cal T}_{X'/S}}\left({\cal A}_{X}^{gp}\otimes_{{\cal O}_{X}}F_{X/S}^{*} \hat{S^{\cdot}}{\cal T}_{X'/S}\right)
\end{equation*}
depending on the choice of a log strong lifting.
\end{theo}
As is mentioned above, since Gros-Le Stum-Quir\'os's dual argument in the proof of (\ref{e5}) strongly depend on the local freeness of the $(m+1)$-st relative Frobenius of $X\to S$,
it is hard to imitate their proof in the logarithmic case.
We will give a simpler proof based on the Azumaya nature of $\tilde{{\cal D}}^{(m)}_{X/S}$.
Then we will obtain our main result (Theorem \ref{t9}), which we call the log local Cartier transform of higher level, by using the indexed version of Morita equivalence due to Schepler.
\begin{theo}
Let $\cal J$ be a sheaf of ${\cal M}^{gp}_{X}/{\cal O}^{*}_{X}$-sets on $X$.
Then there exists an equivalence of categories
\begin{equation*}
C_{\tilde{F}}: 
{\left(
\begin{array}{c}
\text{The category of } {\cal J}\text{-indexed left }\\
 \hat{S^{\cdot}}{\cal T}_{X'/S}\otimes_{S^{.}{\cal T}_{X'/S}}\tilde{{\cal D}}_{X/S}^{(m)} \text{-modules}
\end{array}
\right)}
\xrightarrow{\cong} 
{\left(
\begin{array}{c}
\text{The category of } {\cal J}\text{-indexed left }\\
 {\cal B}_{X/S}^{(m+1)}\otimes_{{\cal O}_{X'}}\hat{S^{\cdot}}{\cal T}_{X'/S} \text{-modules}
\end{array}
\right)}.
\end{equation*}

\end{theo}
We can write the ${\cal D}_{X/S}^{(m)}$-action on $C_{\tilde{F}}^{-1}({\cal E})$ explicitly (Theorem \ref{t12}).
This is a good point for working with a log strong lifting $\tilde{F}$.
As a natural question of the study of ${\cal D}_{X/S}^{(m)}$-modules in characteristic $p$, 
we also prove the compatibility of the log Frobenius descent and the log local Cartier transform (Theorem \ref{t3}).
This is done by studying the behavior of the splitting module with respect to the log Frobenius descent.
Finally, we clarify the relation between the global version of a splitting module constructed in our previous paper \cite{O} and the local version of a splitting module constructed in this paper.
We construct a global version of a splitting module by glueing our local construction and show that it is isomorphic to our previous splitting module.
As a consequence, we know that the log global Cartier transform of higher level can be reconstructed by a glueing of the log Local Cartier transform of higher level. 

\subsection{Overview}
The content of each section is as follows:
In the second section, we recall several notions and terminologies which we often use in this paper.
We also recall important results in our previous paper \cite{O} which are used in later sections.
In the third section, we construct the divided Frobenius $\Psi: F_{X/S}^{*}\Gamma_{.} {\Omega}^{1}_{X'/S}\to {\cal P}_{X/S, (m)}$ mentioned in the subsection \ref{ss1} 
and give an explicit description of it.
In the fourth section, we prove the main result of this paper called the log local Cartier transform of higher level
and give an explicit description of it.
In the fifth section, we discuss the compatibility of the log local Cartier transform and the log Frobenius descent.
In the final section, we construct the global version of a splitting module by the glueing argument and show that this is coincide with a splitting module constructed in \cite{O}.

\subsection{Conventions}
Throughout this paper, we fix a prime number $p$ and a natural number $m$.
We use the following notations on multi-indices.
The element $(0, . . . , 1, . . . , 0)\in {\mathbb N}^{r}$, where 1 sits in the $i$-th entry, is denoted by $\underline{\varepsilon}_{i}$.
When $\underline{k}$ is an element of ${\mathbb N}^{r}$, we denote its $i$-th entry by $k_{i}$, $k_{1}+\cdots +k_{r}$ by $|\underline{k}|$.
We denote a log scheme by a single letter such as $X$.
For a log scheme $X$, we denote the structure sheaf of $X$ by ${\cal O}_{X}$ and the log structure of $X$ by ${\cal M}_{X}$.

\section{Review}\label{s1}
In this section, after recalling basic notations on log differential operators of higher level and indexed algebras associated to log structure, we recall some results in our previous paper \cite{O} needed later.

\subsection{Log differential operators of higher level}\label{a}
In this subsection, we recall basic definitions and notations on log differential operators of higher level introduced by Montagnon.
For more details, see \cite{M}.
We also refer the reader to the subsection 3.1 of \cite{O}.

Let us first recall the definition of $m$-PD structures [B1, D\'efinition 1.3.1].

\begin{defi}
Let $X$ be a log scheme over ${\mathbb Z}_{(p)}$, $I$ a quasi coherent ideal of ${\cal O}_{X}$.  
Then $m$-PD structure on $I$ is a PD ideal $(J, \gamma)$ such that $I^{(p^{m})}+pI\subset J\subset I$ and $\gamma$ is compatible with the standard PD structure on $p{\mathbb Z}_{(p)}$.  
Here $I^{(p^{m})}$ denotes the subideal of $I$ generated by $p^{m}$-th powers of local sanctions of $I$.
\end{defi}
In the case $m=0$, the definition of $m$-PD structures coincides with that of classical PD structures. 
Let $(J, \gamma)$ be an $m$-PD structure on $I$ . For each natural number $k$, we define by $f^{\{k\}}:=f^{r}\gamma_{q}(f^{p^{m}})$, where $k=p^{m}q+r$ and $0\leq r< p^{m}$.

For a while, $X\to S$ denotes a log smooth morphism of fine log schemes over ${\mathbb Z}_{(p)}$ with a suitable $m$-PD structure on a quasi-coherent ideal of ${\cal O}_{S}$ (see [B1, D\'efinition 1.3.2]). 
We assume that $p$ is locally nilpotent on $X$.
Let us take the log $m$-PD envelope $X\hookrightarrow P_{X/S, (m)}$ of the diagonal immersion $X\hookrightarrow X\times_{S}X$, that is,  
$X\hookrightarrow P_{X/S, (m)}$ is the universal object of exact closed immersions with an $m$-PD structure on its defining ideal, which is compatible with the $m$-PD structure on $S$, over $X\hookrightarrow X\times_{S}X$ (for more details, see [M, Proposition 2.1.1]). 
We denote by ${\cal P}_{X/S, (m)}$ the structure sheaf of $P_{X/S, (m)}$.
We also denote by $\bar{I}$ the defining ideal of the exact closed immersion $X\hookrightarrow P_{X/S, (m)}$.
Note that, by definition of log $m$-PD envelope, $\bar{I}$ is endowed with the $m$-PD structure,
so one can take the $m$-PD-adic filtration (for its definition, see \cite{B2} D\'efinition A.3) $\left\{\bar{I}^{\{n\}}\right\}_{n\in {\mathbb N}}$ associated to $\bar{I}$.
For a natural number $n$, we put ${\cal P}^{n}_{X/S, (m)}:={\cal P}_{X/S, (m)}/\bar{I}^{\{n+1\}}$ and ${\cal D}_{X/S, n}^{(m)} :={\cal H}om_{{\cal O}_{X}}({{\cal P}}_{X/S, (m)}^{n}, {\cal O}_{X})$. 
Then we define the sheaf ${\cal D}_{X/S}^{(m)}$ of log differential operators of level $m$ by
\begin{equation*}
{\cal D}_{X/S}^{(m)}:=\bigcup_{n\in \mathbb{N}} {\cal D}_{X/S, n}^{(m)}.
\end{equation*} 
${\cal D}_{X/S}^{(m)}$ forms a sheaf of ${\cal O}_{X}$-algebra via the comultiplication $\delta^{n, n'}: {\cal P}^{n+n'}_{X/S, (m)}\to {\cal P}^{n}_{X/S, (m)}\otimes_{{\cal O}_{X}}{\cal P}^{n'}_{X/S, (m)}$
naturally induced from the projection $X\times_{S}X\times_{S}X\to X\times_{S}X$ to the first and the third factors.

Finally we recall a local description of ${\cal D}_{X/S}^{(m)}$.
Let us denote by $p_{0}$ (resp. $p_{1}$) the first (resp. the second) projection of $P_{X/S, (m)}$.
For any section $a\in {\cal M}_{X}$, there exists the unique section $\mu(a)\in (1+\bar{I})$ such that $p_{1}^{*}(a)=p_{0}^{*}(a)\cdot \mu(a)$. 
We define the section $\eta_{a}\in \bar{I}$ by $\mu(a)-1$.
Log smoothness of $X\to S$ implies that, \'etale locally on $X$, there is a logarithmic system of coordinates $m_{1}, \ldots, m_{r}\in {\cal M}_{X}^{gp}$, that is, a system of sections such that the set $\left\{d\log m_{1},\ldots, d\log m_{r}\right\}$ forms a basis of the log differential module $\Omega_{X/S}^{1}$ of $X$ over $S$. 
We define the section $\eta^{\{\underline{k}\}}$ by $\eta^{\{\underline{k}\}}=\prod_{i:=1}^{r}\eta_{m_{i}}^{\{k_{i}\}}$ for each multi-index $\underline{k}\in {\mathbb N}^{r}$.
Then the set $\left\{ \underline{\eta}^{\{\underline{k}\}} \bigl| |\underline{k}|\leq n\right\}$ forms a local basis of ${\cal P}_{X/S, (m)}^{n}$ over ${\cal O}_{X}$.
We denote the dual basis of $\bigl\{ \underline{\eta}^{\{\underline{k}\}} \bigl| |\underline{k}|\leq n\bigr\}$ by $\bigl\{ \partial_{\langle \underline{k} \rangle} \bigl| |\underline{k}|\leq n\bigr\}$.


\subsection{Indexed algebras associated to log structure}\label{c2}
In this subsection, after recalling basics on indexed modules, we introduce two indexed algebras ${\cal A}_{X}^{gp}$, ${\cal B}^{(m+1)}_{X/S}$ associated to log structures introduced by Lorenzon and Montagnon.
For more details, see \cite{L}, \cite{S} and \cite{M}.
See also the section 2 of \cite{O} for basics on indexed modules and the subsection 4.1 of \cite{O} for basics on ${\cal A}_{X}^{gp}$ and ${\cal B}^{(m+1)}_{X/S}$. 

Let $X$ be a scheme and $\cal I$ a sheaf of abelian groups on $X$. 
Let $a$ be the addition ${\cal I \times I\to I}$ of $\cal I$. 
\begin{defi}
{\rm (1)} an $\cal I$-indexed ${\cal O}_{X}$-module is a sheaf $\cal F$ over $\cal I$ with an addition $\cal F \times_{I}F\to F$ over $\cal I$, a unit $\cal I\to F$ over $\cal I$, an inverse $\cal F\to F$ over $\cal I$ and 
an ${\cal O}_{X}$-action ${\cal O}_{X}\times \cal F\to F$ over $\cal I$ satisfying the usual axiom of ${\cal O}_{X}$-modules. 

{\rm (2)} an $\cal I$-indexed ${\cal O}_{X}$-algebra is an $\cal I$-indexed ${\cal O}_{X}$-module $\cal A$ with an
${\cal O}_{X}$-bilinear multiplication $\cal A \times A \to A$ over $a$, a global section $1_{\cal A}\in {\cal A}$ over the zero section $0 : e\to \cal I$ satisfying the usual axiom of non commutative rings. 
The notion of commutative $\cal I$-indexed ${\cal O}_{X}$-algebra is also defined in a natural way.
\end{defi}
Let $\cal A$ be an $\cal I$-indexed ${\cal O}_{X}$-algebra. 
Let $\cal J$ be a sheaf of $\cal I$-sets on $X$ and $\rho$ the action $\cal I\times J\to J$ of $\cal I$ on $\cal J$. 
\begin{defi}
$\cal J$-indexed $\cal A$-module is a  
$\cal J$-indexed ${\cal O}_{X}$-module $\cal E$ with an ${\cal O}_{X}$-linear ${\cal A}$-action ${\cal A\times E\to E}$ over $\rho$ satisfying the usual axiom of modules over a ring.
\end{defi}

As is the same with the theory of usual modules over a ring, we can define and construct several notions on indexed modules such as tensor products, internal homomorphisms, local freeness etc.
Furthermore, Schepler proved the indexed version of the Morita equivalence. 
\begin{pro}\label{l}
Let $\cal A$ be a commutative $\cal I$-indexed ${\cal O}_{X}$-algebra and $\cal J$ a sheaf of $\cal I$-sets, 
Let $M$ be a locally free $\cal I$-indexed $\cal A$-module of finite rank. 
We put ${\cal E}nd_{\cal A}(M):=\cal E$ (this naturally forms an $\cal I$-indexed ${\cal O}_{X}$-algebras). 
Then the functor $E \longmapsto M\otimes_{{\cal A}} E$ defines an equivalence of categories between the category of $\cal J$-indexed $\cal A$-modules and the category of $\cal J$-indexed left $\cal E$-modules with a quasi inverse 
$F\longmapsto {\cal H}om(M, F)$. 
\end{pro}
\begin{proof}
See [S, Theorem 2.2].
\end{proof}
Let us define the notion of Azumaya algebra.
\begin{defi}
Let $\cal A$ be a commutative $\cal I$-indexed ${\cal O}_{X}$-algebra and $\cal E$ an $\cal I$-indexed $\cal A$-algebra.
$\cal E$ is an Azumaya algebra over $\cal A$ of rank $r$ if there exists some commutative faithfully flat  $\cal I$-indexed $\cal A$-algebra $\cal B$ and
an $\cal I$-indexed locally free $\cal B$-module $M$ of rank $r$ such that 
\begin{equation*}
{\cal E\otimes_{A}B }\xrightarrow{\cong}{\cal E}nd_{\cal B}(M).
\end{equation*}
Let $\cal E$ be an Azumaya algebra over $\cal A$.
When there exists some commutative $\cal I$-indexed $\cal A$-algebra $\cal C$ such that ${\cal E\otimes_{A}C }\xrightarrow{\cong}{\cal E}nd_{\cal C}(M)$
for some $\cal I$-indexed locally free $\cal C$-module $M$,
we call $M$ a splitting module of $\cal E$ over $\cal C$.
\end{defi}

Let us introduce ${\cal I}_{X}^{gp}$-indexed ${\cal O}_{X}$-algebra ${\cal A}_{X}^{gp}$ which can be considered as a generalization of the structure sheaf of $X$ in some sense. 
Let $X$ be a fine log scheme. 
We denote by ${\cal I}_{X}^{gp}$ the quotient sheaf ${\cal M}_{X}^{gp}/{\cal O}_{X}^{*}$ of abelian groups.
Here ${\cal M}_{X}^{gp}$ denotes the group associated to ${\cal M}_{X}$.  
As an \'etale sheaf, we define ${\cal A}_{X}^{gp}:=({\cal M}_{X}^{gp}\times{\cal O}_{X})/\sim$, where  $\sim$ denotes an equivalence relation defined by $(ax, y)\sim (x, ay)$ for any  $a\in {\cal O}^{*}_{X}$, $x\in {\cal M}_{X}^{gp}$ and  $y\in {\cal O}_{X}$. 
We then regard ${\cal A}_{X}^{gp}$ as an \'etale sheaf over ${\cal I}_{X}^{gp}$ via the natural morphism induced by a composition ${\cal M}_{X}^{gp}\times {\cal O}_{X}\to {\cal M}_{X}^{gp}\to {\cal I}_{X}^{gp}$, where the first morphism is the first projection and the second one is the natural projection.
Furthermore,  ${\cal A}_{X}^{gp}$ naturally forms an ${\cal I}_{X}^{gp}$-indexed ${\cal O}_{X}$-algebra via the ${\cal O}_{X}$-action on the second entry and a multiplication induced from the addition of ${\cal M}_{X}^{gp}$.

Let $X\to S$ be a log smooth morphism of fine log schemes over ${\mathbb Z}_{(p)}$ with a suitable $m$-PD structure on a quasi-coherent ideal of ${\cal O}_{S}$. 
One can prove the ${\cal O}_{X}$-action on ${\cal A}_{X}^{gp}$ naturally extend to the ${\cal D}^{(m)}_{X/S}$-action and this action satisfies the Leibniz type formulas. 
Therefore the scalar extension
\begin{equation*}
\tilde{{\cal D}}^{(m)}_{X/S}:={\cal A}^{gp}_{X}\otimes_{{\cal O}_{X}}{\cal D}^{(m)}_{X/S}
\end{equation*}
naturally forms an ${\cal I}^{gp}_{X}$-indexed ${\cal A}^{gp}_{X}$-algebra. 
From the rest of this paper, we will study the Azumaya nature of $\tilde{{\cal D}}^{(m)}_{X/S}$ instead of ${\cal D}^{(m)}_{X/S}$.

From now on, let $X\to S$ be a log smooth morphism of fine log schemes defined over ${\mathbb Z}/p{\mathbb Z}$ and we endow
$S$ with the $m$-PD structure on the zero ideal. 
Let us consider the following commutative diagram: 

\[\xymatrix{
{X}\ar[r] \ar@/^1pc/[rrr]^{F_{X}}\ar@/_/[rrd] & {X'}\ar[r]&\ar@{}[rd]|{\square} {X''} \ar[d] \ar[r] & X \ar[d] \\
& & S \ar[r]^{F_{S}} & S ,}
\]

where $F_{X}$ (resp. $F_{S}$) denotes the $(m+1)$-st composition of the absolute Frobenius of $X$ (resp. $S$), the right square is cartesian and the above diagram $X\to X'\to X''$
denotes a unique factorization by a purely inseparable morphism $X\to X'$ and a log \'etale morpsiam $X'\to X''$ (see [K, Proposition 4.10 (2)]).
We denote $X\to X'$ by $F_{X/S}$ and call it the $(m+1)$-st relative Frobenius of $X\to S$. 
We denote the composition $X'\to X''\to X$ by $\pi$. 
We  denote $X'$ by $X^{(m+1)}$, if there is a risk of confusion (especially we will use this notation in section \ref{s5}). 
As an ${\cal I}_{X}^{gp}$-indexed sheaf of abelian groups, we define ${\cal B}^{(m+1)}_{X/S}$ by
\begin{equation*}
{\cal B}_{X/S}^{(m+1)}:= {\cal H}om_{{\cal D}^{(m+1)}_{X/S}}\left(F_{X/S}^{*}{\cal D}_{X'/S}^{(0)}, {\cal A}_{X}^{gp}\right).
\end{equation*}
For the definition of left ${\cal D}^{(m+1)}_{X/S}$-action on $F_{X/S}^{*}{\cal D}_{X'/S}^{(0)}$, see [M, Chapitre 3] or the subsection 4.1.2 of \cite{O}. 
If we define a morphism $i:{\cal B}_{X/S}^{(m+1)}\rightarrow {\cal A}_{X}^{gp}$ by $g\mapsto g(1\otimes1)$, 
then this is injective and one can check that ${\cal B}_{X/S}^{(m+1)}$ forms an indexed subalgebra of ${\cal A}_{X}^{gp}$. 
Furthermore we define ${\cal O}_{X'}$-action on ${\cal B}_{X/S}^{(m+1)}$ by the right multiplication of ${\cal O}_{X'}$ on $F^{*}{\cal D}_{X'/S}^{(0)}$.
Then ${\cal B}_{X/S}^{(m+1)}$ forms an ${\cal I}_{X}^{gp}$-indexed ${\cal O}_{X'}$-algebra. 
The following theorem proved by Montagnon is important. 
\begin{theo}\label{e}
${\cal A}_{X}^{gp}$ is a locally free ${\cal I}_{X}^{gp}$-indexed ${\cal B}_{X/S}^{(m+1)}$-module of rank $p^{r(m+1)}$. 
\end{theo}
\begin{proof}
See [M, Corollaire 4.2.1].
\end{proof}

\subsection{Azumaya algebra property}\label{j}

In this subsection, we recall the Azumaya nature of $\tilde {{\cal D}}_{X/S}^{(m)}$. 
For more details, see the sections 3 and 4 of \cite{O}. 
Let $X\to S$ be as in the end of the subsection \ref{c2}.

First we recall the $p^{m+1}$-curvature map which generalizes the classical $p$-curvature map.
\begin{theo} \label{t2}
There exists a unique morphism of ${\cal O}_{X}$-algebras $\beta: S^{\cdot} {\cal T}_{X'/S}\to {\cal D}^{(m)}_{X/S}$ which sends $\xi^{'}_{i}$ to $\underline{\partial}_{\langle p^{m+1}\underline{\varepsilon}_{i}\rangle}$,
where $\left\{\xi_{i}^{'}\bigl| 1\leq i\leq r\right\}$ denotes the dual basis of  $\left\{\pi^{*}d\log m_{i}\bigl| 1\leq i\leq r\right\}$ with a logarithmic system of coordinates $\left\{m_{i}| 1\leq i\leq r\right\}$ of $X\to S$.
\end{theo}
\begin{proof}
For the construction of the $p^{m+1}$-curvature map, see [O, Definition 3.10] and its local description, see [O, Proposition 3.11].
\end{proof}
In the case of the trivial log structure and $m=0$, the $p^{m+1}$-curvature map coincides with the classical $p$-curvature map. 

\begin{theo} \label{t1}
{\rm (1)}
Let $\tilde{\mathfrak Z}$ denote the center of $\tilde {{\cal D}}_{X/S}^{(m)}$.
Then $\beta: {S^{\cdot} {\cal T}_{X'/S}}\to {\cal D}_{X/S}^{(m)}$ induces an isomorphism between ${\cal B}_{X/S}^{(m+1)}\otimes_{{\cal O}_{X'}}S^{\cdot} {\cal T}_{X'/S}$ and $\tilde{\mathfrak Z}$ as an indexed subalgebra of $\tilde {{\cal D}}_{X/S}^{(m)}$.

{\rm (2)}
The ${\cal I}_{X}^{gp}$-indexed ${\cal O}_{X}$-algebra $\tilde{{\cal D}}_{X/S}^{(m)}$ is an Azumaya algebra over its center $\tilde{\mathfrak Z}$ of rank $p^{r(m+1)}$.
\end{theo}
\begin{proof}
See [O, Theorem 4.16]  for a proof of $(1)$ and [O, Corollary 4.20]  for a proof of $(2)$.
\end{proof}

\section{The divided Frobenius map}

In this section, we construct the divided Frobenius map $\Psi: F_{X/S}^{*}\Gamma_{.} {\Omega}^{1}_{X'/S}\to {\cal P}_{X/S, (m)}$ (see (\ref{d6})) which
is essential for a construction of a splitting module of our Azumaya algebra.
Throughout this section, $X\to S$ denotes an integral log smooth morphism of fine log schemes defined over ${\mathbb Z}/p{\mathbb Z}$ and
and we endow $S$ with the $m$-PD structure on the zero ideal. 
To simplify the argument, we assume that the underlying scheme $S$ is noetherian and $X\to S$ is of finite type.
We essentially use this assumption in the construction of $\Psi$ but we expect that we can construct $\Psi$ without this assumption.
We freely use terminologies introduced in the section \ref{s1}.

We first confirm the definition of a lifting modulo $p^{2}$ in this paper.
\begin{defi}
Let $f:Y\to Z$ be a morphism of fine log schemes defined over ${\mathbb Z}/p{\mathbb Z}$.
Then a lifting of $f$ modulo $p^{2}$ is a morphism $\tilde{f}: \tilde{Y}\to \tilde{Z}$ of fine log schemes flat over ${\mathbb Z}/p^{2}{\mathbb Z}$ which fits into a cartesian square in the category of fine log schemes
	\begin{equation*}
		\begin{CD}
			Y @>>>\tilde{Y}\\
			@VVV @VVV\\
			Z @>>> \tilde{Z},\\
		\end{CD}
	\end{equation*}
where $Z\to \tilde{Z}$ is the exact closed immersion defined by $p$.
\end{defi}
We state the following lemmas needed later.
\begin{lemm}\label{l3}
Let $M$ be a ${\mathbb Z}/p^{2}{\mathbb Z}$-module.
Then multiplication by $p!$ induces a surjective homomorphism
\begin{equation*}
p!: M/pM \to pM
\end{equation*}
which is an isomorphism if $M$ is flat over ${\mathbb Z}/p^{2}{\mathbb Z}$.
\end{lemm}
\begin{proof}
We omit the proof.
\end{proof}
\begin{lemm}\label{l4}
We have for any $m>0$
\begin{equation*}
\binom{p^{m+1}}{i}=\left\{
		\begin{array}{c}
		\text{$1$} \\
		\text{$\frac{(-1)^{k}p!}{k}$}\\
		\text{$0$ }
		\end{array}
		\right.\,\,\,\,\,\,
		\begin{array}{l}
		\text{if $i=0$ or $i=p^{m+1}$}\\
		\text{if $i=kp^{m}$}\\
		\text{othewise}
		\end{array}
		\,\,\,\,\,\,
		\text{mod } p^{2}
\end{equation*}

\end{lemm}
\begin{proof}
See [GLQ, Lemma 4.3].
\end{proof}
\begin{rem}
It seems that there is a misprint in [GLQ, Lemma 4.3].
The author would like to thank Kazuaki Miyatani for pointing him out this misprint and teaching him its correct version.
\end{rem}

Next, let us introduce the notion of log strong lifting of the $(m+1)$-st relative Frobenius of $X\to S$ which is the log version of strong lifting defined in \cite{GLQ} (see [GLQ, Definition 4.4]).
\begin{defi}\label{d1}
Let $\tilde{F}: \tilde{X}\to \tilde{X'}$ be a lifting of $F_{X/S}$ mod $p^{2}$.
We say $\tilde{F}$ is a log strong lifting if, for any $m\in {\cal M}_{X}$ with a lifting $\tilde{m}\in {\cal M}_{\tilde{X}}$ of $m$ and any lifting $\tilde{m'}\in {\cal M}_{\tilde{X'}}$ of $\pi^{*}m\in {\cal M}_{X'}$,
there exists $\tilde{g}\in {\cal O}_{\tilde{X}}$ such that
\begin{equation}\label{e3}
\tilde{F}^{*}(\tilde{m'})=\tilde{m}^{p^{m+1}}(1+p\tilde{g}^{p^{m}}).
\end{equation}    
\end{defi}
\begin{rem}
In the case $m=0$, any lifting of $F_{X/S}$ modulo $p^{2}$ is a log strong lifting. 
\end{rem}
We give a basic property of log strong lifting.

\begin{lemm}\label{r1}
Let $\tilde{F}$ be a log strong lifting.
For any $f\in {\cal O}_{X}$ with a lifting $\tilde{f}\in {\cal O}_{\tilde{X}}$ of $f$, there exists a lifting $\tilde{f'}\in {\cal O}_{\tilde{X'}}$ of $1\otimes f\in {\cal O}_{X'}$
satisfying
\begin{equation}\label{e4}
\tilde{F}^{*}(\tilde{f'})=\tilde{f}^{p^{m+1}} +p\tilde{g}^{p^{m}}\text{\,\,\,\,\,\,\,\,\,\,\,\,for some $\tilde{g}\in {\cal O}_{\tilde{X}}$}.
\end{equation}    
\end{lemm}
\begin{proof}
We may work locally on $X$.
In the case $f\in {\cal O}^{*}_{X}$, the assertion is obvious.
In the case $f \notin  {\cal O}^{*}_{X}$, $1+f\in {\cal O}^{*}_{X}$.
So we can write $\tilde{F}^{*}(\tilde{f'})=(1+\tilde{f})^{p^{m+1}} +p\tilde{g}^{p^{m}}$ for some lifting $\tilde{f'}\in {\cal O}_{\tilde{X'}}$ of $\left(1\otimes1+1\otimes f\right)$  and $\tilde{g}\in {\cal O}_{\tilde{X}}$.
Then $\tilde{f'}-1$ is a lifting of $1\otimes f$ and, by Lemma \ref{l4}, we have
\begin{equation*}
\tilde{F}^{*}(\tilde{f'}-1)=(1+\tilde{f})^{p^{m+1}} +p\tilde{g}^{p^{m}}-1=\tilde{f}^{p^{m+1}}+p\tilde{f}^{p^{m}}+p\tilde{g}^{p^{m}}=\tilde{f}^{p^{m+1}}+p(\tilde{f}+\tilde{g})^{p^{m}}
\end{equation*}
in modulo $p^{2}$.
This finishes the proof.
\end{proof}

The following lemma gives a typical example of log strong liftings.
\begin{lemm}\label{l1}
Let $\tilde{F}_{X}$ be a lifting of the absolute Frobenius of $X$.
Then, for any $\tilde{m}\in {\cal M}_{\tilde{X}}$, there exists $\tilde{g}\in {\cal O}_{\tilde{X}}$ such that
\begin{equation*}
\tilde{F}_{X}^{n+1*}(\tilde{m})=\tilde{m}^{p^{n+1}}(1+p\tilde{g}^{p^{n}}) \text{\,\,\,\,\,\,\,\,\,\,\,\,for any $n\in {\mathbb N}$}.
\end{equation*}
\end{lemm}
\begin{proof}
We show it induction on $n$.
In the case when $n=1$, we have $\tilde{F}_{X}^{*}\left(\tilde{m}\right)=\tilde{m}^{p}\tilde{u}$ with $\tilde{u}\in {\cal O}^{*}_{\tilde{X}}$ satisfying $q^{*}\tilde{u}=1$.
By this, $\tilde{u}=1+p\tilde{g}$ for some $\tilde{g}\in {\cal O}_{\tilde{X}}$.
Hence we have $\tilde{F}_{X}^{*}(\tilde{m})=\tilde{m}^{p}(1+p\tilde{g})$.
Let us assume $\tilde{F}_{X}^{n+1*}(\tilde{m})=\tilde{m}^{p^{n+1}}(1+p\tilde{g}^{p^{n}})$.
Then, by Lemma \ref{r1}, 
\begin{eqnarray*}
\tilde{F}_{X}^{n+2*}(\tilde{m})&=&\tilde{F}_{X}^{*}\left(\tilde{m}^{p^{n+1}}(1+p\tilde{g}^{p^{n}})\right)\\
&=& \left\{\tilde{m}^{p}(1+p\tilde{g})\right\}^{p^{n+1}}\left(1+\tilde{F}_{X}^{*}\left( p\tilde{g}^{p^{n}}\right)\right)\\
&=& \tilde{m}^{p^{n+2}} (1+p\tilde{g})^{p^{n+1}} \left(1+p \left(\tilde{g}^{p}+p\tilde{h} \right)^{p^{n}}\right)\\
&=& \tilde{m}^{p^{n+2}} (1+p\tilde{g}^{p^{n+1}}).
\end{eqnarray*}
We finish the proof.
\end{proof}
\begin{rem} \label{r2}
Since $X\to S$ is log smooth, a lifting $\tilde{F}_{X}$ of the absolute Frobenius of $X$ always exists \'etale locally on $X$.
So, by Lemma \ref{l1}, we can always take a log strong lifting \'etale locally on $X$. 
\end{rem}
In the rest of this section, we fix a log strong lifting $\tilde{F}$.
Now we may start to construct the divided Frobenius $\Psi: F_{X/S}^{*}\Gamma_{.} {\Omega}^{1}_{X'/S}\to {\cal P}_{X/S, (m)}$.
Let $\tilde{X}\hookrightarrow P_{\tilde{X}/\tilde{S}, (m)}$ be the log $m$-PD envelope of the diagonal $\tilde{X}\to {\tilde{X}\times_{\tilde{S}}\tilde{X}}$.
Let us consider the following commutative diagram:
\[\xymatrix{
 {\tilde{X}} \ar[rr]^{\tilde{F}}\ar[d] & & {\tilde{X'}}\ar[d] \\
 {P_{\tilde{X}/\tilde{S}, (m)}} \ar[r] & {\tilde{X}\times_{\tilde{S}}\tilde{X}} \ar[r]^{\tilde{F}\times \tilde{F}} & {\tilde{X'}\times_{\tilde{S}}\tilde{X'}}
  .}\] 
By [Proposition 4.10, K], \'etale locally on $\tilde{X'}$, $\tilde{X'}\hookrightarrow \tilde{X'}\times_{\tilde{S}}\tilde{X'}$ factors as $\tilde{X'}\hookrightarrow Z \to \tilde{X'}\times_{\tilde{S}}\tilde{X'}$.
Here the first map is an exact closed immersion and the second one is log \'etale.
Then, since the defining ideal $\bar{I}$ of $\tilde{X}\hookrightarrow {P_{\tilde{X}/\tilde{S}, (m)}}$ is endowed with the $m$-PD structure, so $\tilde{X}\hookrightarrow P_{\tilde{X}/\tilde{S}, (m)}$ is nilimmersion and
the morphism $P_{\tilde{X}/\tilde{S}, (m)} \to \tilde{X'}\times_{\tilde{S}}\tilde{X'}$ uniquely factors as $P_{\tilde{X}/\tilde{S}, (m)} \to Z \to \tilde{X'}\times_{\tilde{S}}\tilde{X'}$ (cf, [NS, Lemma 2.3.14]).
Since the underlying scheme of $Z$ is locally noetherian and $\tilde{X}\hookrightarrow P_{\tilde{X}/\tilde{S}, (m)}$ is nilimmersion, $P_{\tilde{X}/\tilde{S}, (m)} \to Z$ uniquely factors as $P_{\tilde{X}/\tilde{S}, (m)} \to \tilde{X'}(N) \to Z$ for a sufficiently large number $N$, where $\tilde{X'}(N)$ denotes the $N$-th log infinitesimal neighborhood of $\tilde{X'}\hookrightarrow Z$.
Then, since $\tilde{X'}\hookrightarrow \tilde{X'}(N)$ has the universality, we  obtain the morphism $\tilde{\Psi}_{N}: P_{\tilde{X}/\tilde{S}, (m)} \to \tilde{X'}(N)$ for some large number $N$ globally on $\tilde{X'}$.
Here we use the quasi-compactness of $\tilde{X'}$ to bound such $N$'s.
 
Let us calculate the image of ${\cal I}:=Ker({\cal O}_{\tilde{X'}(N)}\to {\cal O}_{\tilde{X'}})$ under $\tilde{\Psi}^{*}_{N}: {\cal O}_{\tilde{X'}(N)}\to {{\cal P}_{\tilde{X}/\tilde{S}, (m)}}$. 
Let $m$ be a local section of ${\cal M}_{X}$ with a lift $\tilde{m}\in {\cal M}_{\tilde{X}}$.
Let $\tilde{m'}\in {\cal M}_{\tilde{X'}}$ be a lift of $\pi^{*}m$. 
Then, by the definition of log strong lifting, we can write $\tilde{F}^{*}\left(\tilde{m'}\right)=\tilde{m}^{p^{m+1}}\left(1+p\tilde{g}^{p^{m}}\right)$ for some $\tilde{g}\in {\cal O}_{\tilde{X}}$.
On the other hand, there exists a unique section $\mu_{(\infty)}(\tilde{m'})\in {\cal M}_{\tilde{X'}(N)}$ satisfying $q_{1}^{*}\left(\tilde{m'}\right)=q_{0}^{*}\left(\tilde{m'}\right)\cdot\mu_{(\infty)}(\tilde{m'})$
(cf, subsection \ref{a}).
Note that $\mu_{(\infty)}(\tilde{m'})-1$ is contained in $\cal I$.
Under these notations (see also subsection \ref{a}), we give the following calculation.
\begin{lemm}\label{l2}
We have
\begin{equation*}
\tilde{\Psi}^{*}_{N}(\mu_{(\infty)}(\tilde{m'})-1)=p!\left\{\eta_{\tilde{m}}^{\{p^{m+1}\}} + \sum_{k=1}^{p-1}\frac{(-1)^{p^{m}k}}{k}\eta_{\tilde{m}}^{k}+ \left(1\otimes \tilde{g}-\tilde{g}\otimes1\right)^{p^{m}}\right\}.
\end{equation*}
\end{lemm}
\begin{proof}
Since the following diagram is commutative
\[\xymatrix{
 {P_{\tilde{X}/\tilde{S}}} \ar[r]^{\Psi}\ar[d] & {\tilde{X'}(N)} \ar[d]\\
 {\tilde{X}\times_{\tilde{S}}\tilde{X}} \ar@<1ex>[d]^{p_{1}}\ar@<-1ex>[d]_{p_{0}}\ar[r]^{\tilde{F}\times \tilde{F}} & {\tilde{X'}\times_{\tilde{S}}\tilde{X'}}\ar@<1ex>[d]^{q_{1}}\ar@<-1ex>[d]_{q_{0}}\\
 {\tilde{X}} \ar[r]^{\tilde{F}} & {\tilde{X'}}
 ,}\]
 we have $\tilde{\Psi}^{*}\left(\mu_{(\infty)}(\tilde{m'})\right)= \mu\left(\tilde{F}^{*}\left(\tilde{m'}\right)\right)=\mu\left(\tilde{m}^{p^{m+1}}\left(1+p\tilde{g}^{p^{m}}\right)\right)=\mu\left(\tilde{m}\right)^{p^{m+1}}\mu\left(1+p\tilde{g}^{p^{m}}\right)$.
We have
\begin{eqnarray*}
\mu\left(1+p\tilde{g}^{p^{m}}\right)&=&p_{0}^{*}\left(1+p\tilde{g}^{p^{m}}\right)^{-1}p_{1}^{*}\left(1+p\tilde{g}^{p^{m}}\right)\\
&=&\left(1-p\tilde{g}^{p^{m}}\right)\otimes \left(1+p\tilde{g}^{p^{m}}\right)\\
&=&1\otimes1+p\left(1\otimes \tilde{g}-\tilde{g}\otimes1\right)^{p^{m}}.
\end{eqnarray*}
Therefore, by Lemma \ref{l4} and the fact $p! \equiv -p$ mod $p^{2}$, we have
\begin{eqnarray*}
\Psi^{*}\left(\mu_{(\infty)}(\tilde{m})-1\right)&=&  \mu(\tilde{m'})^{p^{m+1}}\left(1\otimes1+p\left(1\otimes \tilde{g}-\tilde{g}\otimes1\right)^{p^{m}}\right)-1\\
&=& \left(\eta_{\tilde{m}} +1\right)^{p^{m+1}}\left(1\otimes1+p\left(1\otimes \tilde{g}-\tilde{g}\otimes1\right)^{p^{m}}\right)-1 \\
&=& \left(\eta_{\tilde{m}}^{p^{m+1}}+\sum_{k=1}^{p^{m+1}-1}\binom {p^{m+1}}{k} \eta_{\tilde{m}}^{k}+1\right) \left(1\otimes1+p\left(1\otimes \tilde{g}-\tilde{g}\otimes1\right)^{p^{m}}\right)-1\\
&=& \left(p!\eta_{\tilde{m}}^{\{p^{m+1}\}}+\sum_{k=1}^{p-1}\frac{(-1)^{p^{m}k}}{k}p!\eta_{\tilde{m}}^{k}+1\right) \left(1\otimes1+p\left(1\otimes \tilde{g}-\tilde{g}\otimes1\right)^{p^{m}}\right)-1\\
&=& p!\left\{\eta_{\tilde{m}}^{\{p^{m+1}\}} + \sum_{k=1}^{p-1}\frac{(-1)^{p^{m}k}}{k}\eta_{\tilde{m}}^{k}+ \left(1\otimes \tilde{g}-\tilde{g}\otimes1\right)^{p^{m}}\right\}.
\end{eqnarray*}

\end{proof}

By Lemma \ref{l2}, $\Psi^{*}$ sends $\cal I$ into $p!{\cal P}_{\tilde{X}/\tilde{S}, (m)}$.
We define the morphism ${\cal I}\to {\cal P}^{n}_{X/S}$ by the composition
\begin{equation*}
{\cal I}\xrightarrow{\tilde{\Psi}_{N}^{*}} p!{\cal P}_{\tilde{X}/\tilde{S}, (m)} \xrightarrow{\cong} p{\cal P}_{\tilde{X}/\tilde{S}, (m)}\xrightarrow{\cong} {\cal P}_{\tilde{X}/\tilde{S}, (m)}/p{\cal P}_{\tilde{X}/\tilde{S}, (m)}\xrightarrow{\cong}
{\cal P}_{X/S, (m)},
\end{equation*}
where the first isomorphism is induced by ${\rm id}_{{\cal P}_{\tilde{X}/\tilde{S}, (m)}}$ (hence $p!x$ maps to $-px$) and the second isomorphism is the one in Lemma \ref{l3}.
Since the image of $\cal I$ under this map is contained in the ideal generated by $p$, the image of ${\cal I}^{2}$ under this map is zero. 
We thus obtain the morphism ${\Omega}^{1}_{\tilde{X'}/\tilde{S}}\to {\cal P}^{n}_{X/S}$.
Again by Lemma \ref{l2}, the image of $p{\Omega}^{1}_{\tilde{X'}/\tilde{S}}$ under this map is zero.
We thus obtain the morphism ${\Omega}^{1}_{X'/S}\to {\cal P}^{n}_{X/S}$.
Again by Lemma \ref{l2}, the image of ${\Omega}^{1}_{X'/S}$ under this map is contained in the underlying PD-ideal of $\bar{I}$.
Therefore we have the morphism of PD algebras 
\begin{equation}\label{d6}
\Psi: F^{*}_{X/S}\Gamma_{.}{\Omega}^{1}_{X'/S}\to {\cal P}_{X/S, (m)}.
\end{equation}
We call $\Psi$ the divided Frobenius map.
For each $n$, we define $\Psi_{n}: F^{*}_{X/S}\Gamma_{.}{\Omega}^{1}_{X'/S}\to {\cal P}^{n}_{X/S}$ by the composition of morphisms 
\begin{equation}\label{d6}
\Psi_{n}: F^{*}_{X/S}\Gamma_{.}{\Omega}^{1}_{X'/S}\xrightarrow{\Psi} {\cal P}_{X/S, (m)} \to {\cal P}^{n}_{X/S},
\end{equation}
where the second morphism is the natural projection.
By taking the ${\cal O}_{X}$-duals of $\Psi$ and $\Psi_{n}$, we obtain 
\begin{equation*}
\Phi_{n}: {\cal D}_{X/S, n}^{(m)}\to {\cal H}om_{{\cal O}_{X}}(F_{X/S}^{*}\Gamma_{.}{\Omega}^{1}_{X'/S}, {\cal O}_{X})\xrightarrow{\cong} F_{X/S}^{*}\hat{S^{\cdot}}{\cal T}_{X'/S}, 
\end{equation*}
where the first morphism is the ${\cal O}_{X}$-dual of $\Psi^{*}$ and $\hat{S^{\cdot}}{\cal T}_{X'/S}$ denotes the completion of $S^{\cdot}{\cal T}_{X'/S}$ by the augmentation ideal,
and
\begin{equation*}
\Phi: {\cal H}om_{{\cal O}_{X}}({\cal P}_{X/S, (m)}, {\cal O}_{X})\to F_{X/S}^{*}\hat{S^{\cdot}}{\cal T}_{X'/S}.
\end{equation*}

To give the local description of $\Phi_{n}$, let us set up some notations.
Let us assume we are given a logarithmic system of coordinates $\left\{m_{i}| 1\leq i\leq r\right\}$ of $X\to S$.
Then $\left\{\pi^{*}m_{i}| 1\leq i\leq r\right\}$ forms a logarithmic system of coordinates of $X'\to S$.
Let take a lifting $\tilde{m'}_{i}$ of $\pi^{*}m_{i}$.
Then, since $\tilde{F}$ is a log strong lifting,
we can write 
\begin{equation*}
\tilde{F}^{*}(\tilde{m'}_{i})=\tilde{m}_{i}^{p^{m+1}}(1+p\tilde{g}_{i}^{p^{m}}) \text{\,\,\,\,\,\,\,\,\,\,\,\,for some $\tilde{g}_{i}\in {\cal O}_{\tilde{X}}$}.
\end{equation*}
for each $i$.
We put $g_{i}:=\tilde{g}_{i}$ modulo $p\in {\cal O}_{X}$.
Let $\{\underline{\partial}_{\langle \underline{k}\rangle} \}$ be an ${\cal O}_{X}$-basis of ${\cal D}^{(m)}_{X/S}$ associated to $\{m_{i} \}$
and $\left\{\xi_{i}^{'}\bigl| 1\leq i\leq r\right\}$ the dual basis of  $\left\{\pi^{*}d\log m_{i}\bigl| 1\leq i\leq r\right\}$.
Under these notations, we give the following calculation which is the log version of [GLQ, Proposition 4.10].

\begin{pro}\label{p3}
For any $n\geq p^{m}$, we have
\begin{equation*}
\Phi_{n}\left(\partial_{\langle\underline{k}\rangle}\right)=\left\{
		\begin{array}{c}
		\text{$1$} \\
		\text{$0$}\\
		\text{$\left(1-\partial_{i}(g_{i})^{p^{m}}\right)\xi'_{i} -\sum_{j=1, j\neq i}^{r}\partial_{j}(g_{j})^{p^{m}}\xi'_{j}$ }
		\end{array}
		\right.
		\begin{array}{l}
		\text{if \,\,\,\,\,$|\underline{k}|=0$}\\
		\text{if \,\,\,\,\,$0<|\underline{k}|< p^{m}$}\\
		\text{if \,\,\,\,\,$\underline{k}=p^{m}\underline{\varepsilon}_{i}$,}
		\end{array}
\end{equation*}
where $\partial_{j}$ denotes $\partial_{\langle \underline{\varepsilon}_{j}\rangle}$.
\end{pro}
\begin{proof}
In the case $|\underline{k}|=0$ the assertion is obvious.
By Lemma \ref{l2} and the construction of $\Psi_{n}: F^{*}_{X/S}\Gamma_{\cdot}\Omega^{1}_{X'/S}\to {\cal P}^{n}_{X/S, (m)}$, $\Psi_{n}$ sends $1\otimes d\log \pi^{*}m_{j}$ to
\begin{equation}\label{e6}
-\eta_{m_{j}}^{\{p^{m+1}\}} - \sum_{k=1}^{p-1}\frac{(-1)^{p^{m}k}}{k}\eta_{m_{j}}^{p^{m}k}-\left(1\otimes g_{j}-g_{j}\otimes1\right)^{p^{m}}.
\end{equation}
By this description, we see the assertion in the case $0<|\underline{k}|< p^{m}$.
Let us consider the case of $\underline{k}=p^{m}\underline{\varepsilon}_{i}$.
We need to calculate the image of (\ref{e6}) under $\partial_{\langle \underline{\varepsilon}_{i}\rangle}$, so
we may ignore $\eta_{m_{j}}^{\{\underline{l}\}}$ with $|\underline{l}|> p^{m}$.
We thus calculate 
\begin{eqnarray*}
(\ref{e6})&=& -\eta_{m_{j}}^{\{p^{m+1}\}} - \sum_{k=1}^{p-1}\frac{(-1)^{p^{m}k}}{k}\eta_{m_{j}}^{p^{m}k}-\left(\sum_{\underline{k}\neq 0} \partial_{\langle \underline{k}\rangle}(g_{j})\eta^{\{\underline{k}\}} \right)^{p^{m}}\\
&=&\eta_{m_{j}}^{p^{m}}-\sum_{j=1}^{r}\partial_{j}\left( g_{j}\right)^{p^{m}}\eta_{m_{j}}^{p^{m}}+ \text{ (the terms of $\eta_{m_{j}}^{\{\underline{l}\}}$ with $|\underline{l}|> p^{m}$)}.\\
\end{eqnarray*}
By this description, we see that $\partial_{\langle \underline{\varepsilon}_{i}\rangle}$ sends (\ref{e6}) to $-\left(1+\partial_{i}(g_{j})^{p^{m}}\right)$ if $i=j$
and to $\partial_{j}\left( g_{j}\right)^{p^{m}}$ if $i\neq j$.
We have seen the image of $\left(1\otimes d\log \pi^{*}m_{j}\right)^{\left[\underline{l}\right]}$ under $\Psi_{n}$ is zero if $|\underline{l}|>1$.
Hence $\Phi_{n}\left(\partial_{\langle p^{m}\underline{\varepsilon}_{i}\rangle}\right)\in \hat{S^{\cdot}}{\cal T}_{X'/S}$ is
\begin{equation*}
\left(1-\partial_{i}(g_{i})^{p^{m}}\right)\xi'_{i}-\sum_{j=1, j\neq i}^{r}\partial_{j}\left( g_{j}\right)^{p^{m}}\xi'_{j}.
\end{equation*}
We finish the proof.
\end{proof}


\section{Local Cartier transform}
In this section, we construct the log Local Cartier transform of higher level.
Throughout this section, assume that we are given an integral log smooth morphism $X\to S$ defined over ${\mathbb Z}/p{\mathbb Z}$ with a log strong lifting $\tilde{F}$ and
a sheaf of ${\cal I}^{gp}_{X}$-sets $\cal J$.

\begin{defi}
For each $n\in {\mathbb N}$, we define the morphism 
$\rho_{n}: {\cal P}^{n}_{X/S, (m)}\otimes_{{\cal O}_{X}}F_{X/S}^{*}{\Gamma}_{\cdot} \Omega^{1}_{X'/S} \to F_{X/S}^{*}{\Gamma}_{\cdot} \Omega^{1}_{X'/S}\otimes_{{\cal O}_{X}}{\cal P}^{n}_{X/S, (m)}$ of ${\cal P}^{n}_{X/S, (m)}$-modules by the ${\cal P}^{n}_{X/S, (m)}$-linearization of the composition
\begin{equation*}
\rho'_{n}: F_{X/S}^{*}{\Gamma}_{\cdot} \Omega^{1}_{X'/S}\xrightarrow{\Delta} F_{X/S}^{*}{\Gamma}_{\cdot} \Omega^{1}_{X'/S}\otimes_{{\cal O}_{X}}F_{X/S}^{*}{\Gamma}_{\cdot} \Omega^{1}_{X'/S}\xrightarrow{{\rm id}\otimes {\Psi}_{n}}F_{X/S}^{*}{\Gamma}_{\cdot} \Omega^{1}_{X'/S}\otimes_{{\cal O}_{X}} {\cal P}^{n}_{X/S, (m)},
\end{equation*}
where $\Delta$ is the comultiplication of $F_{X/S}^{*}{\Gamma}_{\cdot} \Omega^{1}_{X'/S}$ (which is the ${\cal O}_{X}$-dual of the multiplication of $F^{*}_{X/S}\hat{S^{\cdot}}{\cal T}_{X'/S}$).
\end{defi}

\begin{pro}\label{p1}
$\{\rho_{n} \}$ forms an $m$-PD stratification on $F_{X/S}^{*}{\Gamma}_{\cdot} \Omega^{1}_{X'/S}$.
\end{pro}
\begin{proof}
By [M, Proposition 2.6.1], it suffices to show that $\{\rho'_{n}\}$ satisfies $\rho'_{1}={\rm id}$ and the following diagram is commutative:

\[\xymatrix{
{F_{X/S}^{*}{\Gamma}_{\cdot} \Omega^{1}_{X'/S}}\ar[r]^{\rho'_{n+n'}} \ar[d]_{\rho'_{n}}& {F_{X/S}^{*}{\Gamma}_{\cdot} \Omega_{X'/S}\otimes_{{\cal O}_{X}} {\cal P}^{n+n'}_{X/S, (m)}}\ar[d]^{{\rm id}\otimes \delta^{n, n'}}\\
{F_{X/S}^{*}{\Gamma}_{\cdot} \Omega_{X'/S}\otimes_{{\cal O}_{X}} {\cal P}^{n}_{X/S, (m)} }\ar[r] & {F_{X/S}^{*}{\Gamma}_{\cdot} \Omega_{X'/S}\otimes_{{\cal O}_{X}} {\cal P}^{n'}_{X/S, (m)}\otimes_{{\cal O}_{X}}{\cal P}^{n}_{X/S, (m)}}
,}\]
where the lower horizontal arrow is $\rho'_{n'}\otimes {\rm id}$.
$\rho'_{1}={\rm id}$ is obvious.
The commutativity of the above diagram follows from the commutativity of the following diagram:
\[\xymatrix{
{F_{X/S}^{*}{\Gamma}_{\cdot} \Omega^{1}_{X'/S}} \ar[r]^{\Psi_{n+n'}} \ar[d]_{\Delta} & {{\cal P}^{n+n'}_{X/S, (m)}}\ar[d]^{\delta^{n, n'}}\\
{F_{X/S}^{*}{\Gamma}_{\cdot} \Omega^{1}_{X'/S}\otimes_{{\cal O}_{X}}F_{X/S}^{*}{\Gamma}_{\cdot} \Omega^{1}_{X'/S}} \ar[r] & {{\cal P}^{n'}_{X/S, (m)}\otimes_{{\cal O}_{X}}{\cal P}^{n}_{X/S, (m)},}
}\]
where the lower horizontal arrow is $\Psi_{n'}\otimes\Psi_{n}$.
This fact follows from our geometric construction of $\Psi$.

\end{proof}
By Proposition \ref{p1}, $F_{X/S}^{*}\Gamma_{.}{\Omega}^{1}_{X'/S}$ has the ${\cal D}^{(m)}_{X/S}$-action associated to $\{\rho_{n} \}$.
By taking the dual as left ${\cal D}^{(m)}_{X/S}$-modules,
we obtain the ${\cal D}^{(m)}_{X/S}$-action on $F_{X/S}^{*} \hat{S^{\cdot}}{\cal T}_{X'/S}$ .
On the other hand, let us endow $F_{X/S}^{*}\hat{S}^{\cdot}{\cal T}_{X'/S}$ with an $\hat{S}^{\cdot}{\cal T}_{X'/S}$-action as follows.
We first consider the morphism of ${\cal O}_{X}$-modules defined by
\begin{equation*}
S^{\cdot}{\cal T}_{X'/S}\to {\cal D}^{(m)}_{X/S}\to {\cal H}om_{{\cal O}_{X}}({\cal P}_{X/S, (m)}, {\cal O}_{X}),
\end{equation*}
where the first morphism is the $p^{m+1}$-curvature map (see Theorem \ref{t2}) and the second one is the inductive limit of the dual of the natural projection
${\cal P}_{X/S, (m)}\to {\cal P}^{n}_{X/S, (m)}$.
This naturally extends to the morphism of ${\cal O}_{X}$-modules $\hat{S}^{\cdot}{\cal T}_{X'/S}\to {\cal H}om_{{\cal O}_{X}}({\cal P}_{X/S, (m)}, {\cal O}_{X})$.
We define $\hat{S}^{\cdot}{\cal T}_{X'/S}$-action on $F_{X/S}^{*}\hat{S}^{\cdot}{\cal T}_{X'/S}$ by
\begin{eqnarray*}
\hat{S}^{\cdot}{\cal T}_{X'/S}\otimes_{{\cal O}_{X}} F_{X/S}^{*}\hat{S}^{\cdot}{\cal T}_{X'/S}&\xrightarrow{\beta\otimes{\rm id}}& {\cal H}om_{{\cal O}_{X}}({\cal P}_{X/S, (m)}, {\cal O}_{X})\otimes_{{\cal O}_{X}} F_{X/S}^{*}\hat{S}^{\cdot}{\cal T}_{X'/S}\\
&\xrightarrow{\Phi\otimes {\rm id}}& F_{X/S}^{*}\hat{S}^{\cdot}{\cal T}_{X'/S}\otimes_{{\cal O}_{X}}F_{X/S}^{*}\hat{S}^{\cdot}{\cal T}_{X'/S}\\
&\longrightarrow & F_{X/S}^{*}\hat{S}^{\cdot}{\cal T}_{X'/S},
\end{eqnarray*}
where the third arrow denotes the multiplication.
We now obtain the following splitting isomorphism.
\begin{theo}\label{t8}
There exists an action of $\left({\cal B}_{X/S}^{(m+1)}\otimes_{{\cal O}_{X'}}\hat{S^{\cdot}}{\cal T}_{X'/S}\right) \otimes_{{\cal B}_{X/S}^{(m+1)}\otimes_{{\cal O}_{X'}}S^{\cdot}{\cal T}_{X'/S}}\tilde{{\cal D}}^{(m)}_{X/S}$ on ${\cal A}_{X}^{gp}\otimes_{{\cal O}_{X}}F_{X/S}^{*} \hat{S^{\cdot}}{\cal T}_{X'/S}$, which depends on a choice of a log strong lifting $\tilde{F}$ such that it defines an isomorphism of ${\cal I}^{gp}_{X}$-indexed algebras
\begin{equation*}
\left({\cal B}_{X/S}^{(m+1)}\otimes_{{\cal O}_{X'}}\hat{S^{\cdot}}{\cal T}_{X'/S}\right) \otimes_{{\cal B}_{X/S}^{(m+1)}\otimes_{{\cal O}_{X'}}S^{\cdot}{\cal T}_{X'/S}}\tilde{{\cal D}}^{(m)}_{X/S}\xrightarrow{\cong}
{\cal E}nd_{{\cal B}_{X/S}^{(m+1)}\otimes_{{\cal O}_{X'}}\hat{S^{\cdot}}{\cal T}_{X'/S}}\left({\cal A}_{X}^{gp}\otimes_{{\cal O}_{X}}F_{X/S}^{*} \hat{S^{\cdot}}{\cal T}_{X'/S}\right).
\end{equation*}
\end{theo}
\begin{proof}
By construction, we see that our ${\cal D}^{(m)}_{X/S}$-action on $F_{X/S}^{*} \hat{S^{\cdot}}{\cal T}_{X'/S}$ naturally extends to $\hat{S^{\cdot}}{\cal T}_{X'/S} \otimes_{S^{\cdot}{\cal T}_{X'/S}}{\cal D}^{(m)}_{X/S}$-action on it.
Thus ${\cal A}_{X}^{gp}\otimes_{{\cal O}_{X}}F_{X/S}^{*} \hat{S^{\cdot}}{\cal T}_{X'/S}$ has an action of
\begin{equation*}
{\cal A}_{X}^{gp}\otimes_{{\cal O}_{X}}\hat{S^{\cdot}}{\cal T}_{X'/S} \otimes_{S^{\cdot}{\cal T}_{X'/S}}{\cal D}^{(m)}_{X/S} \xrightarrow{\cong} \left({\cal B}_{X/S}^{(m+1)}\otimes_{{\cal O}_{X'}}\hat{S^{\cdot}}{\cal T}_{X'/S}\right) \otimes_{{\cal B}_{X/S}^{(m+1)}\otimes_{{\cal O}_{X'}}S^{\cdot}{\cal T}_{X'/S}}\tilde{{\cal D}}^{(m)}_{X/S}
\end{equation*}
and we obtain the morphism of ${\cal I}^{gp}_{X}$-indexed algebras in Theorem \ref{t8}.
By Theorem \ref{e}, ${\cal A}_{X}^{gp}\otimes_{{\cal O}_{X}}F_{X/S}^{*} \hat{S^{\cdot}}{\cal T}_{X'/S}\xrightarrow{\cong}{\cal A}_{X}^{gp}\otimes_{{\cal O}_{X'}} \hat{S^{\cdot}}{\cal T}_{X'/S}$ is locally free over ${\cal B}_{X/S}^{(m+1)}\otimes_{{\cal O}_{X'}}\hat{S^{\cdot}}{\cal T}_{X'/S}$ of rank $p^{r(m+1)}$
and, by Theorem \ref{t1} (2), ${\cal B}_{X/S}^{(m+1)}\otimes_{{\cal O}_{X'}}\hat{S^{\cdot}}{\cal T}_{X'/S} \otimes_{{\cal B}_{X/S}^{(m+1)}\otimes_{{\cal O}_{X'}}S^{\cdot}{\cal T}_{X'/S}}\tilde{{\cal D}}^{(m)}_{X/S}$ is an Azumaya algebra of rank $p^{r(m+1)}$ over ${\cal B}_{X/S}^{(m+1)}\otimes_{{\cal O}_{X'}}\hat{S^{\cdot}}{\cal T}_{X'/S}$.
Therefore this morphism is an isomorphism of ${\cal I}^{gp}_{X}$-indexed algebras by [S, Corollary 2.5].
\end{proof}

By virtue of Proposition \ref{l}, we now obtain our main theorem in this paper, which we call the log local Cartier transform of higher level. 

\begin{theo}\label{t9}
The functor ${\cal F}\mapsto {\cal H}om({\cal A}^{gp}_{X}\otimes_{{\cal O}_{X}}F_{X/S}^{*}\hat{S^{\cdot}}{\cal T}_{X'/S}, {\cal F})$ induces an equivalence of categories
\begin{equation*}
C_{\tilde{F}}: \left(
		\begin{array}{c}
		\text{The category of $\cal J$-indexed} \\
		\text{left $\hat{S^{\cdot}}{\cal T}_{X'/S}\otimes_{S^{.}{\cal T}_{X'/S}}\tilde{{\cal D}}_{X/S}^{(m)}$-modules on $X$}
		\end{array}
	\right)
	\xrightarrow{\cong}
	\left(
		\begin{array}{c}
		\text{The category of $\cal J$-indexed} \\
		\text{${\cal B}^{(m+1)}_{X/S}\otimes_{{\cal O}_{X'}} \hat{S^{\cdot}}{\cal T}_{X'/S}$-modules on $X'$}
		\end{array}
	\right)
\end{equation*}
with a quasi-inverse $C_{\tilde{F}}^{-1}: {\cal E}\mapsto {\cal A}^{gp}_{X}\otimes_{{\cal O}_{X}}F_{X/S}^{*}\hat{S^{\cdot}}{\cal T}_{X'/S}\otimes {\cal E}$.
\end{theo}
\begin{rem}
Theorem \ref{t9} is regarded as a generalization of [OV, Theorem 2.8] and [GLQ, Theorem 5.8] to the case of log schemes and that of [S, Theorem 4.16] to the case of higher level.
\end{rem}

Let us describe the ${\cal D}_{X/S}^{(m)}$-action on the Cartier transform $C_{\tilde{F}}^{-1}\left({\cal E} \right)$ with a $\cal J$-indexed ${\cal B}^{(m+1)}_{X/S}\otimes_{{\cal O}_{X'}} \hat{S^{\cdot}}{\cal T}_{X'/S}$-module $\cal E$ explicitly.
Note that $\cal E$ can be regarded as the $\cal J$-indexed ${\cal B}^{(m+1)}_{X/S}$-module $\cal E$ with the morphism of indexed algebras
\begin{equation*}
\theta: \hat{S^{\cdot}}{\cal T}_{X'/S}\to {\cal E}nd_{{\cal B}^{(m+1)}_{X/S}}({\cal E}).
\end{equation*}
As a ${\cal J}$-indexed ${\cal A}^{gp}_{X}$-module, $C_{\tilde{F}}^{-1}\left({\cal E} \right)$ is isomorphic to ${\cal A}^{gp}_{X}\otimes_{{\cal B}^{(m+1)}_{X/S}}{\cal E}$.
We want to describe the ${\cal D}_{X/S}^{(m)}$-action on ${\cal A}^{gp}_{X}\otimes_{{\cal B}^{(m+1)}_{X/S}}{\cal E}$ by using $\theta$.
Let $\left\{m_{i}| 1\leq i\leq r\right\}$ of $X\to S$ be a logarithmic system of coordinates of $X\to S$.
Let $\{\underline{\partial}_{\langle \underline{k}\rangle} \}$, $\left\{\xi_{i}^{'}\bigl| 1\leq i\leq r\right\}$ and $\{\tilde{g}_{i}\}$ be as in Proposition \ref{p3}.
Then $\underline{\partial}_{\langle p^{s}\underline{\varepsilon}_{i}\rangle}$ with $0\leq s\leq m$ and $1\leq i\leq r$ acts on $1\otimes1\otimes e\in {\cal A}^{gp}_{X}\otimes_{{\cal O}_{X}}F_{X/S}^{*}\hat{S^{\cdot}}{\cal T}_{X'/S}\otimes {\cal E}$ by
\begin{eqnarray*}
\underline{\partial}_{\langle p^{s}\underline{\varepsilon}_{i}\rangle}. (1\otimes1\otimes e) &=& 1\otimes \Phi_{n}\left(\underline{\partial}_{\langle p^{s}\underline{\varepsilon}_{i}\rangle}\right)\otimes e\\
\end{eqnarray*}
for a sufficiently large $n\in {\mathbb N}$.
Therefore, by Proposition \ref{p3}, we have the following formulas.
\begin{theo}\label{t12}
Under the notation above, $\underline{\partial}_{\langle p^{s}\underline{\varepsilon}_{i}\rangle}$ acts on $1\otimes e\in {\cal A}^{gp}_{X}\otimes_{{\cal B}^{(m+1)}_{X/S}}{\cal E}$ by
\begin{equation*}
\underline{\partial}_{\langle p^{s}\underline{\varepsilon}_{i}\rangle}. (1\otimes e)=\left\{
		\begin{array}{c}
		\text{$0$} \\
		\text{$1\otimes\left\{ \left(1-\partial_{i}(g_{i})^{p^{m}}\right)\theta\left(\xi'_{i}\right)(e) -\sum_{j=1, j\neq i}^{r}\partial_{j}(g_{j})^{p^{m}}\theta\left(\xi'_{j}\right)(e)\right\}$}
		\end{array}
		\right.
		\begin{array}{l}
		\text{if \,\,\,\,\,$s<m$}\\
		\text{if \,\,\,\,\,$s=m$.}
		\end{array}
\end{equation*}

\end{theo}

Finally, we give a variant of the log local Cartier transform of level $0$.
Let $F: X\to X^{(m)}$ denote the $m$-th relative Frobenius of $X\to S$ (see the subsection \ref{c2}).
Let $\tilde {{\cal D}}_{X^{(m)}/S}^{(0)}$ denote the ${\cal I}_{X}^{gp}$-indexed ${\cal O}_{X^{(m)}}$-module ${\cal B}_{X/S}^{(m)}\otimes_{{\cal O}_{X^{(m)}}}{\cal D}_{X^{(m)}/S}^{(0)}$.
Then there exists a natural ${\cal I}_{X}^{gp}$-indexed ${\cal O}_{X^{(m)}}$-algebra structure on $\tilde {{\cal D}}_{X^{(m)}/S}^{(0)}$, for more details see [M, Lemma 4.2.1].
Note that $\tilde {{\cal D}}_{X^{(m)}/S}^{(0)}$ is not equal to ${\cal A}^{gp}_{X^{(m)}}\otimes_{{\cal O}_{X^{(m)}}}{\cal D}_{X^{(m)}/S}^{(0)}$ since ${\cal B}_{X/S}^{(m)}\neq {\cal A}^{gp}_{X^{(m)}}$ in general (see [L, 1.8] for a counter-example).
As in the case of $\tilde {{\cal D}}_{X/S}^{(m)}$, we had the following theorem in \cite{O}.

\begin{theo}\label{t4}
Let $X\to S$ be a log smooth morphism of fine log schemes.

{\rm (1)}
${\cal B}_{X/S}^{(m)}$ is locally free as an ${\cal I}_{X}^{gp}$-indexed ${\cal B}_{X/S}^{(m+1)}$-module of rank $p^{r}$.

{\rm (2)}
Let $\tilde{\mathfrak Z'}$ denote the center of $\tilde {{\cal D}}_{X^{(m)}/S}^{(0)}$.
Then the $p$-curvature map ${\mathfrak Z}\to {\cal D}_{X^{(m)}/S}^{(0)}$ induces an isomorphism between ${\cal B}_{X/S}^{(m+1)}\otimes_{{\cal O}_{X'}}\mathfrak Z$ and $\tilde{\mathfrak Z'}$ as an indexed subalgebra of $\tilde {{\cal D}}_{X^{(m)}/S}^{(0)}$.

{\rm (3)}
The ${\cal I}_{X}^{gp}$-indexed ${\cal O}_{X^{(m)}}$-algebra $\tilde{{\cal D}}_{X^{(m)}/S}^{(0)}$ is an Azumaya algebra over its center $\tilde{\mathfrak Z'}$ of rank $p^{2r}$.
\end{theo}
\begin{proof}
See [O, Proposition 4.14] for the proof of (1).
See [O, Theorem 4.21] for the proof of (2) and (3).
\end{proof}
Let $\tilde{F'}$ be a lifting of the first relative Frobenius $F': X^{(m)}\to X^{(m+1)}=X'$ of $X^{(m)}\to S$.
Then, by the similar argument in  Theorem \ref{t8}, one can see that there exists an isomorphism of ${\cal I}^{gp}_{X}$-indexed algebras
\begin{equation*}
{\cal B}_{X/S}^{(m)}\otimes_{{\cal O}_{X'}}\hat{S^{\cdot}}{\cal T}_{X'/S} \otimes_{{\cal B}_{X/S}^{(m+1)}\otimes_{{\cal O}_{X'}}S^{\cdot}{\cal T}_{X'/S}}\tilde{{\cal D}}^{(0)}_{X^{(m)}/S}\xrightarrow{\cong}
{\cal E}nd_{{\cal B}_{X/S}^{(m+1)}\otimes_{{\cal O}_{X'}}\hat{S^{\cdot}}{\cal T}_{X'/S}}\left({\cal B}_{X/S}^{(m)}\otimes_{{\cal O}_{X^{(m)}}}F'^{*} \hat{S^{\cdot}}{\cal T}_{X'/S}\right)
\end{equation*}
by Theorem \ref{t4} (1), (2) and (3).
Then we have a variant of the log local Cartier transform of level $0$ (see also [S, Theorem 4.16]).
\begin{theo}
The functor ${\cal F}\mapsto {\cal H}om({\cal B}^{(m)}_{X/S}\otimes_{{\cal O}_{X^{(m)}}}F'^{*}\hat{S^{\cdot}}{\cal T}_{X'/S}, {\cal F})$ induces an equivalence of categories
\begin{equation*}
C'_{\tilde{F'}}: \left(
		\begin{array}{c}
		\text{The category of $\cal J$-indexed} \\
		\text{left $\hat{S^{\cdot}}{\cal T}_{X'/S}\otimes_{S^{.}{\cal T}_{X'/S}}\tilde{{\cal D}}_{X^{(m)}/S}^{(0)}$-modules on $X^{(m)}$}
		\end{array}
	\right)
	\xrightarrow{\cong}
	\left(
		\begin{array}{c}
		\text{The category of $\cal J$-indexed} \\
		\text{${\cal B}^{(m+1)}_{X/S}\otimes_{{\cal O}_{X'}} \hat{S^{\cdot}}{\cal T}_{X'/S}$-modules on $X'$}
		\end{array}
	\right)
\end{equation*}
with a quasi-inverse $C_{\tilde{F}}^{-1}: {\cal E}\mapsto {\cal B}^{(m)}_{X/S}\otimes_{{\cal O}_{X^{(m)}}}F'^{*}\hat{S^{\cdot}}{\cal T}_{X'/S}\otimes {\cal E}$.
\end{theo}

\section{Frobenius descent}\label{s5}
In this section, we discuss the compatibility of the log local Cartier transform and the log Frobenius descent.
Let us first briefly recall Montagnon's log Frobenius descent.
Let $X\to S$ be a log smooth morphism of fine log schemes defined over ${\mathbb Z}/p{\mathbb Z}$.
Let $F: X\to X^{(m)}$ be the $m$-th relative Frobenius of $X\to S$.
We fix a sheaf of ${\cal I}^{gp}_{X}$-sets $\cal J$ and let $\cal F$ be a $\cal J$-indexed left $\tilde {{\cal D}}_{X^{(m)}/S}^{(0)}$-module.
We consider ${\cal J}$-indexed sheaf defined by
	\begin{equation*}
	{\mathbb G}_{\cal J}({\cal F}):={\cal A}_{X}^{gp}\otimes_{F^{*}{\cal B}_{X/S}^{(m)}} F^{*}{\cal F} .
	\end{equation*}
Then we can endow ${\mathbb G}_{\cal J}({\cal F})$ with a ${\cal J}$-indexed $\tilde{\cal D}_{X/S}^{(m)}$-module structure and obtain a functor
\begin{equation*}
{\mathbb G}_{\cal J}:
 \left(
\begin{array}{c}
\text{The category of} \\
{\cal J}\text{-indexed left } \tilde {{\cal D}}_{X^{(m)}/S}^{(0)} \text{-modules}

\end{array}
\right)
\to 
\left(
\begin{array}{c}
\text{The category of} \\
{\cal J}\text{-indexed left } \tilde {{\cal D}}_{X/S}^{(m)} \text{-modules}
\end{array}
\right).
\end{equation*}
For more details for the construction of ${\mathbb G}_{\cal J}$ , we refer the reader to the subsection 6.1 in \cite{O}.
	\begin{theo}
	${\mathbb G}_{\cal J}$ is an equivalence of categories.
	\end{theo}
\begin{proof}
For a proof, see Th\'eor\`eme 4.2.1 of \cite{M}.
\end{proof}
We call (a quasi-inverse of) ${\mathbb G}_{\cal J}$ the log Frobenius descent.
\begin{lemm}
	The functor ${\mathbb G}_{\cal J}$ induces an equivalence of categories between 	
	\begin{equation*}
{\mathbb G}_{\cal J}:
\left(
\begin{array}{c}
\text{The category of } {\cal J}\text{-indexed left } \\
 \hat{S^{\cdot}}{\cal T}_{X'/S}\otimes_{S^{.}{\cal T}_{X'/S}} \tilde {{\cal D}}_{X^{(m)}/S}^{(0)} \text{-modules}
\end{array}
\right)
\to
 \left(
\begin{array}{c}
\text{The category of } {\cal J}\text{-indexed left }\\
 \hat{S^{\cdot}}{\cal T}_{X'/S}\otimes_{S^{.}{\cal T}_{X'/S}}\tilde{{\cal D}}_{X/S}^{(m)} \text{-modules}
\end{array}
\right).
\end{equation*}	
\end{lemm}
\begin{proof}
The proof is the same as [O, Lemma 6.6].
\end{proof}
Let us state our main result in this subsection.
Let $X\to S$ be an integral log smooth morphism of fine log schemes defined over ${\mathbb Z}/p{\mathbb Z}$
and we endow $S$ with the $m$-PD structure on the zero ideal. 
We assume that the underlying scheme $S$ is noetherian and $X\to S$ is of finite type.
Let $F$ be the $m$-th relative Frobenius of $X\to S$ and $F'$ the first relative Frobenius of $X^{(m)}\to S$.
Then $F'\circ F$ is nothing but the $(m+1)$-st relative Frobenius $F_{X/S}$ of $X\to S$.
Assume that we are given the following diagram:
\[\xymatrix{
{X} \ar[r]^{F}\ar[d] & {X^{(m)}} \ar[r]^{F'}\ar[d] & {X'}\ar[d]\\
{\tilde{X}\ar[r]} & {\tilde{X^{(m)}}} \ar[r]^{\tilde{F'}} & {\tilde{X'}},
}\]
where $\tilde{X}\to \tilde{X^{(m)}}$ and $\tilde{F'}: \tilde{X^{(m)}} \to \tilde{X'}$ are  liftings of $X\to X^{(m)}$ and $X^{(m)} \to X'$ modulo $p^{2}$ respectively.
We also assume $\tilde{F}: \tilde{X}\to \tilde{X^{(m)}}\to \tilde{X'}$ is a log strong lifting.

\begin{rem}
In the case when the absolute Frobenius $F_{S}^{*}: {\cal O}_{S}\to {\cal O}_{S}$ of $S$ is surjective, one can prove that, if we are given the above diagram, then
$\tilde{F}: \tilde{X}\to \tilde{X^{(m)}}\to \tilde{X'}$ is automatically a log strong lifting.
\end{rem}
Now we may state our main result in this subsection.

\begin{theo}\label{t3}
	Let $\cal J$ be a sheaf of ${\cal I}_{X}^{gp}$-sets on $X$.
	Then the following diagram of categories commutes:
	\[\xymatrix{
	{\left(
\begin{array}{c}
\text{The category of } {\cal J}\text{-indexed left } \\
 \hat{S^{\cdot}}{\cal T}_{X'/S}\otimes_{S^{.}{\cal T}_{X'/S}} \tilde {{\cal D}}_{X^{(m)}/S}^{(0)} \text{-modules}
\end{array}
\right)
} \ar[d]_{{\mathbb G}_{\cal J}} \ar[r]^{C'_{\tilde{F'}}} &  {\left(
\begin{array}{c}
\text{The category of } {\cal J}\text{-indexed left }\\
 {\cal B}_{X/S}^{(m+1)}\otimes_{{\cal O}_{X'}}\hat{S^{\cdot}}{\cal T}_{X'/S} \text{-modules}
\end{array}
\right)
}\\
 {\left(
\begin{array}{c}
\text{The category of } {\cal J}\text{-indexed left }\\
 \hat{S^{\cdot}}{\cal T}_{X'/S}\otimes_{S^{.}{\cal T}_{X'/S}}\tilde{{\cal D}}_{X/S}^{(m)} \text{-modules}
\end{array}
\right)}\ar[ru]_{C_{\tilde{F}}} }.
	\]
	\end{theo}
\begin{rem}
Theorem \ref{t3} can be regarded as the log version of the result stated in Subsection 6.6 of \cite{GLQ}.
For the global version of Theorem \ref{t3}, see [O, Theorem 6.9].
\end{rem}
	
One can prove Theorem \ref{t3} in the same way as the proof of [O, Theorem 6.7]. 
It suffices to show the following lemma.
\begin{lemm}
	The image of ${\cal B}^{(m)}_{X/S}\otimes_{{\cal O}_{X^{(m)}}}F'^{*} \hat{S^{\cdot}}{\cal T}_{X'/S}$ under ${\mathbb G}_{{\cal I}_{X}^{gp}}$ is naturally isomorphic to ${\cal A}_{X}^{gp}\otimes_{{\cal O}_{X}}F_{X/S}^{*} \hat{S^{\cdot}}{\cal T}_{X'/S}$,
	that is, 
	\begin{equation*}
	{\mathbb G}_{{\cal I}_{X}^{gp}}({\cal B}^{(m)}_{X/S}\otimes_{{\cal O}_{X^{(m)}}}F'^{*} \hat{S^{\cdot}}{\cal T}_{X'/S})={\cal A}_{X}^{gp}\otimes_{{\cal O}_{X}}F_{X/S}^{*} \hat{S^{\cdot}}{\cal T}_{X'/S}.
	\end{equation*}
\end{lemm}
\begin{proof}
It suffices to show that $F^{*}F'^{*}\Gamma_{.}{\Omega}_{X'/S}$ is isomorphic to $F_{X/S}^{*}\Gamma_{.}{\Omega}_{X'/S}$ as ${\cal D}_{X/S}^{(m)}$-modules.
We can naturally identify these two objects as ${\cal O}_{X}$-modules.
So we need to check that these two objects have the same ${\cal D}_{X/S}^{(m)}$-action.
Let $\tilde{X}\hookrightarrow {P_{\tilde{X}/\tilde{S}, (m)}}$ (resp. $\tilde{X^{(m)}}\hookrightarrow {P_{\tilde{X^{(m)}}/\tilde{S}, (0)}}$) be the log $m$-PD envelope of the diagonal of $\tilde{X}$ over $\tilde{S}$ (resp. the log $0$-PD envelope of the diagonal of $\tilde{X^{(m)}}$ over $\tilde{S}$).
Let us recall the ${\cal D}^{(m)}_{X/S}$-action on $F^{*}F'^{*}\Gamma_{.}{\Omega}_{X'/S}$.
There exists the unique morphism $F_{\Delta}: {P_{\tilde{X}/\tilde{S}, (m)}}\to {P_{\tilde{X^{(m)}}/\tilde{S}, (0)}}$ (see [M, Subsection 3.3]).
Then the ${\cal D}^{(m)}_{X/S}$-action on $F^{*}F'^{*}\Gamma_{.}{\Omega}_{X'/S}$ is defined by the log $m$-stratification, which is defined by the pull-back of the log $0$-stratification associated to the ${\cal D}^{(0)}_{X^{(m)}/S}$-action on $F'^{*}\Gamma_{.}{\Omega}_{X'/S}$ via the modulo $p$ reduction of $F_{\Delta}$.
So it suffices to prove the equality $\tilde{\Psi}_{N}=\tilde{\Psi'}_{N}\circ F_{\Delta}$ for a sufficiently large number $N$.
Here $\tilde{\Psi}_{N} :P_{\tilde{X}/\tilde{S}, (m)}\to \tilde{X'}(N)$ (resp. $\tilde{\Psi'}_{N}:P_{\tilde{X^{(m)}}/\tilde{S}, (0)}\to \tilde{X'}(N)$) is as in the section $3$.
By definition of $\tilde{\Psi}_{N}$, $\tilde{\Psi}_{N}$ is the unique morphism which makes the diagram 
\[\xymatrix{
{\tilde{X}} \ar[r]\ar[d]& {\tilde{X^{(m)}}}\ar[r] & {\tilde{X'}}\ar[d]\\
{P_{\tilde{X}/\tilde{S}, (m)}} \ar[r]^{\tilde{\Psi}_{N}}& {\tilde{X'}(N)} \ar[r] & {\tilde{X'}\times_{\tilde{S}}{\tilde{X'}}}
}
\]
commutative.
On the other hand, by definition of $F_{\Delta}$ and $\tilde{\Psi'}_{N}$, the following diagram is commutative
\[\xymatrix{
{\tilde{X}} \ar[r]\ar[d]& {\tilde{X^{(m)}}}\ar[d]\ar[r] & {\tilde{X'}}\ar[rd]\\
{P_{\tilde{X}/\tilde{S}, (m)}} \ar[r]^{F_{\Delta}} & {P_{\tilde{X^{(m)}}/\tilde{S}, (0)}} \ar[r]^{\tilde{\Psi'}_{N}}& {\tilde{X'}(N)}\ar[r] & {\tilde{X'}\times_{\tilde{S}}{\tilde{X'}}}
.}
\]
By the above characterization of $\tilde{\Psi}_{N}$, $\tilde{\Psi'}_{N}\circ F_{\Delta}$ must be $\tilde{\Psi}_{N}$.
\end{proof}

\section{Local-global compatibility}
Let $X\to S$ be an integral log smooth morphism defined over ${\mathbb Z}/p{\mathbb Z}$ with a lifting $\tilde{X'}\to \tilde{S}$
of $X'\to S$ modulo $p^{2}$.
We denote $(X\to S, \tilde{X'}\to \tilde{S})$ by $\cal X/S$.
Let $\hat{\Gamma}.{\cal T}_{X'/S}$ be the completion of the PD algebra ${\Gamma}.{\cal T}_{X'/S}$.
In the previous paper \cite{O}, we constructed a splitting module $\check{{\cal K}}^{(m), {\cal A}}_{\cal X/S}$ of the Azumaya algebra $\tilde{\cal D}^{(m)}_{X/S}$ over $\hat{\Gamma}.{\cal T}_{X'/S}\otimes_{{\cal O}_{X'}}{{\cal B}^{(m+1)}_{X/S}}$ (see [O, (5.6)]).
In this section, we give another construction of $\check{{\cal K}}^{(m), {\cal A}}_{\cal X/S}$
under the existence of a lifting $\tilde{X'}\to \tilde{S} \leftarrow \tilde{X}$ of the diagram $X'\to S \leftarrow X$ modulo $p^{2}$.

\subsection{Review of the previous paper}
In this subsection, we briefly recall the construction of $\check{{\cal K}}^{(m), {\cal A}}_{\cal X/S}$ given in the subsection 5.2 in \cite{O}.
For more details, see the section 5 in \cite{O}.
Assume that we are given a lifting of the diagram $X'\to S \leftarrow X$ modulo $p^{2}$.
We define an \'etale sheaf ${\cal L}^{(m)}_{{\cal X/S}}$ on $\tilde{X}$ by 
\begin{equation*}
\text{for each \'etale open $\tilde{U}$ of $\tilde{X}$}\mapsto \text{\{$f: \tilde{U}\to \tilde{X'}$ such that $f$ mod $p=F_{X/S}|_{U}$\}},
\end{equation*}
where $U$ denotes the modulo $p$ reduction of $\tilde{U}$.
Note that in our previous paper, this sheaf is denoted by ${\cal L}^{(m)}_{{\cal X/S}, \tilde{X}}$.
By the deformation theory,  ${\cal L}^{(m)}_{{\cal X/S}}$ forms a
${\cal H}om_{{\cal O}_{X}}(F_{X/S}^{*}\Omega^{1}_{X'/S}, {\cal O}_{X})\cong F^{*}_{X/S}{\cal T}_{X'/S}$-torsor [O, Lemma 5.15]. 
For $a\in {\cal L}^{(m)}_{{\cal X/S}}$ and $\varphi \in {\cal H}om({\cal L}^{(m)}_{{\cal X/S}}, {\cal O}_{X})$, 
we define the morphism of sheaves $\varphi_{a}:F^{*}_{X/S}{\cal T}_{X'/S}\to {\cal O}_{X}$ by $D\mapsto \varphi(a+D)-\varphi(a)$. 
We denote by ${\cal E}_{{\cal X/S}}^{(m)}$ the subsheaf of ${\cal H}om({\cal L}^{(m)}_{{\cal X/S}, \tilde{X}}, {\cal O}_{X})$ consisting of $\varphi:{\cal L}^{(m)}_{{\cal X/S}}\to {\cal O}_{X}$
such that $\varphi_{a}$ is an ${\cal O}_{X}$-homomorphism for any $a\in {\cal L}^{(m)}_{\cal X/S}$.
Then one can check that, for $\varphi\in {\cal E}_{{\cal X/S}}^{(m)}$, $\varphi_{a}$ does not depend on the choice of $a$ and we put $\omega_{\varphi}:=\varphi_{a}$.
We have the following sequence of ${\cal O}_{X}$-modules: 
	\begin{equation}\label{e9}
	\,\,\,\,\,0\longrightarrow {\cal O}_{X}\longrightarrow {\cal E}_{{\cal X/S}}^{(m)}\xrightarrow[\varphi\mapsto \omega_{\varphi}] \ F^{*}_{X/S}{\Omega}^{1}_{X'/S}\longrightarrow 0,
	\end{equation}
where the first map is defined by $b\in {\cal O}_{X} \mapsto {\text {\rm constant function of $b$}}$.
This is a locally splitting exact sequence of ${\cal O}_{X}$-modules.
In fact, one can check exactnesses at ${\cal O}_{X}$ and ${\cal E}_{{\cal X/S}}^{(m)}$ by definition and if we define an ${\cal O}_{X}$-homomorphism $\sigma_{a}: F^{*}_{X/S}{\Omega}^{1}_{X'/S}\to {\cal E}_{{\cal X/S}}^{(m)}$ for $a\in {\cal L}^{(m)}_{{\cal X/S}}$ by $\omega\mapsto [f\mapsto \langle \omega,f-a\rangle]$,
then $\sigma_{a}$ defines a section of ${\cal E}_{{\cal X/S}}^{(m)}\to F^{*}_{X/S}{\Omega}^{1}_{X'/S}$.
For each natural number $n$, we consider $S^{n}({\cal E}_{{\cal X/S}}^{(m)})\hookrightarrow S^{n+1}({\cal E}_{{\cal X/S}}^{(m)})$ induced from ${\cal O}_{X}\hookrightarrow {\cal E}_{{\cal X/S}}^{(m)}$ and
we define the ${\cal O}_{T}$-algebra ${\cal K}^{(m)}_{{\cal X/S}}$ by the inductive limit $\displaystyle \lim_{\to} S^{n}({\cal E}_{{\cal X/S}}^{(m)})$.
Then, by the above locally splitting exact sequence, we have ${\cal E}_{\cal X/S}^{(m)}\simeq {\cal O}_{X}\oplus F^{*}_{X/S}\Omega^{1}_{X'/S}$ in local situation.
We thus locally obtain
\begin{equation*}
\check{{\cal K}}_{\cal X/S}^{(m)}:={\cal H}om_{{\cal O}_{X}}({\cal K}_{\cal X/S}^{(m)}, {\cal O}_{X})\simeq {\cal H}om_{{\cal O}_{X}}(S^{.}F_{X/S}^{*}{\Omega^{1}_{X/S}}, {\cal O}_{X})\simeq \hat{\Gamma.}(F^{*}_{X/S}{\cal T}_{X'/S}).
\end{equation*}
 Therefore $\check{{\cal K}}^{(m), {\cal A}}_{\cal X/S}:={\cal A}_{X}^{gp}\otimes_{{\cal O}_{X}}\check{{\cal K}}^{(m)}_{\cal X/S}$ is a locally free ${\cal I}_{X}^{gp}$-indexed ${\cal A}_{X}^{gp}\otimes_{{\cal O}_{X}}{\hat{\Gamma}_{.}{\cal T}_{X'/S}}$-module of rank $1$.
By Theorem \ref{e}, we see that $\check{{\cal K}}^{(m), {\cal A}}_{\cal X/S}$ is a locally free ${\cal I}_{X}^{gp}$-indexed ${\cal B}^{(m+1)}_{X/S}\otimes_{{\cal O}_{X}}{\hat{\Gamma}_{.}{\cal T}_{X'/S}}$-module of rank $p^{r(m+1)}$.
Furthermore $\check{{\cal K}}^{(m), {\cal A}}_{\cal X/S}$ has an action of 
\begin{equation*}
\left({\cal B}_{X/S}^{(m+1)}\otimes_{{\cal O}_{X'}}\hat{\Gamma.}{\cal T}_{X'/S}\right) \otimes_{{\cal B}_{X/S}^{(m+1)}\otimes_{{\cal O}_{X'}}S^{\cdot}{\cal T}_{X'/S}}\tilde{{\cal D}}^{(m)}_{X/S} \xrightarrow{\cong}{\cal A}_{X}^{gp}\otimes_{{\cal O}_{X}} {\cal D}^{(m)}_{X/S}\otimes_{S^{\cdot}{\cal T}_{X'/S}} \hat{\Gamma.}{\cal T}_{X'/S}
\end{equation*}
(in our previous paper this sheaf is denoted by $\tilde {{\cal D}}_{X/S}^{(m), \gamma}$)
defined as follows.
First of all, we define an action of ${\cal T}_{X'/S}$ on ${\cal K}_{\cal X/S}^{(m)}$ by
\begin{equation*}
{\cal K}^{(m)}_{{\cal X/S}}\simeq S^{\cdot}(F^{*}_{X/S}{\Omega}^{1}_{X'/S})\xrightarrow{D} S^{\cdot}(F^{*}_{X/S}{\Omega}^{1}_{X'/S})\simeq{\cal K}^{(m)}_{{\cal X/S}}.
\end{equation*}
Then one can see that this action is well-defined globally on $X$ and naturally extends to $\hat{\Gamma.}{\cal T}_{X'/S}$-action on ${\cal K}^{(m)}_{{\cal X/S}}$.
On the other hand, we can endow ${\cal K}^{(m)}_{{\cal X/S}}$ with ${\cal D}^{(m)}_{X/S}$-action associated to the structure of log $m$-crystals on it.
Actually sheaves ${\cal L}^{(m)}_{{\cal X/S}}$, ${\cal E}_{\cal X/S}^{(m)}$ and ${\cal K}^{(m)}_{{\cal X/S}}$ naturally extend to log $m$-crystals on ${\rm{Cris}}^{(m)}(X/S)$ (see the subsection 5.2 of \cite{O}).
Then, we can see that these two actions extend to the action of ${\cal D}^{(m)}_{X/S}\otimes_{S^{\cdot}{\cal T}_{X'/S}} \hat{\Gamma.}{\cal T}_{X'/S}$
(see [O, Lemma 5.17]).
Therefore, by the same argument in the proof of Theorem \ref{t9}, we can see $\check{{\cal K}}^{(m), {\cal A}}_{\cal X/S}$ is a splitting module of $\tilde{\cal D}^{(m)}_{X/S}$ over $\hat{\Gamma}.{\cal T}_{X'/S}\otimes_{{\cal O}_{X}}{{\cal A}^{gp}_{X}}$.

\subsection{Glueing}
In this subsection, we give another construction of $\check{{\cal K}}^{(m), {\cal A}}_{\cal X/S}$ based on the maps $\Psi_{n}$ and $\Psi$, which we used to define the log local Cartier transform of higher level.
To do so, it suffices to construct the ${\cal D}^{(m)}_{X/S}$-module ${\cal E}_{{\cal X/S}}^{(m)}$.
Let $X\to S$ be an integral log smooth morphism of fine log schemes defined over ${\mathbb Z}/p{\mathbb Z}$.
We assume that the underlying scheme $S$ is noetherian and $X\to S$ is of finite type.
We also assume that we are given a lifting of the diagram $X'\to S \leftarrow X$ modulo $p^{2}$.

First, assume that we are given a log strong lifting $\tilde{F}$.
For each $n\in {\mathbb N}$, we define the morphism $\epsilon_{n, \tilde{F}}$ of ${\cal P}^{n}_{X/S, (m)}$-modules by the ${\cal P}^{n}_{X/S, (m)}$-linearization of a composition
\begin{equation*}
\epsilon'_{n, \tilde{F}}: {\cal O}_{X}\oplus F_{X/S}^{*}{\Omega_{X'/S}^{1}}\xrightarrow{{\rm id}\otimes \Delta} {\cal O}_{X}\oplus F_{X/S}^{*}{\Omega_{X'/S}^{1}}\otimes_{{\cal O}_{X}}F_{X/S}^{*}{\Omega_{X'/S}^{1}}\xrightarrow{{a}\oplus {\Psi}_{n}\circ i}\left({\cal O}_{X}\oplus F_{X/S}^{*}{\Omega_{X'/S}^{1}}\right)\otimes_{{\cal O}_{X}} {\cal P}^{n}_{X/S, (m)},
\end{equation*}
where $\Delta$ is defined by $x\mapsto x\otimes 1 +1\otimes x$, $a$ is defined by $x\mapsto 1\otimes x$ for $x\in {\cal O}_{X}$ and $i$ is the natural map $F_{X/S}^{*}{\Omega_{X'/S}^{1}}\to F_{X/S}^{*}\Gamma_{\cdot}{\Omega}^{1}_{X'/S}$.
Then $\{\epsilon_{n, \tilde{F}}\}$ forms a log $m$-PD stratification on ${\cal O}_{X}\oplus F_{X/S}^{*}{\Omega_{X'/S}^{1}}$ and we have the left ${\cal D}_{X/S}^{(m)}$-action on ${\cal O}_{X}\oplus F_{X/S}^{*}{\Omega_{X'/S}^{1}}$ associated to $\{\epsilon_{n, \tilde{F}}\}$.

Next, we assume we are given two log strong liftings $\tilde{F}_{1}$ and $\tilde{F}_{2}$ which have the the same domain $\tilde{X}$.
Then we obtain the morphism of ${\cal O}_{X}$-modules $\alpha: F_{X/S}^{*}{\Omega_{X'/S}^{1}}\to {\cal O}_{X}$ characterized by
\begin{eqnarray*}
F^{*}_{X/S}(dx)\mapsto a \text{ where }p\tilde{a}=\tilde{F}_{1}^{*}(\tilde{x})-\tilde{F}_{2}^{*}(\tilde{x})\\
F^{*}_{X/S}(d\log m)\mapsto b \text{ where }\tilde{F}_{2}^{*}(\tilde{m})(1+p\tilde{b})=\tilde{F}_{1}^{*}(\tilde{m}),
\end{eqnarray*}
where $\tilde{a}\in {\cal O}_{\tilde{X}}$ (resp. $\tilde{b}\in {\cal O}_{\tilde{X}}$) is a lift of $a\in {\cal O}_{X}$ (resp. $b\in {\cal O}_{X}$).
We define $u_{12}: {\cal O}_{X}\oplus F_{X/S}^{*}{\Omega_{X'/S}^{1}} \to {\cal O}_{X}\oplus F_{X/S}^{*}{\Omega_{X'/S}^{1}}$ by the composition 
\begin{equation*}
{\cal O}_{X}\oplus F_{X/S}^{*}{\Omega_{X'/S}^{1}}\xrightarrow{0\oplus \alpha}{\cal O}_{X}\hookrightarrow {\cal O}_{X}\oplus F_{X/S}^{*}{\Omega_{X'/S}^{1}},
\end{equation*}
where the second map is the natural inclusion to the first factor.
If we are given another log strong lifting $\tilde{F}_{3}$ of the same domain, then, by construction, $u_{ij}$ satisfies the relation $u_{13}=u_{12}+u_{23}$.
We now define a morphism of ${\cal O}_{X}$-modules $\phi_{12}: {\cal O}_{X}\oplus F_{X/S}^{*}{\Omega_{X'/S}^{1}}\to {\cal O}_{X}\oplus F_{X/S}^{*}{\Omega_{X'/S}^{1}}$ by 
$\phi_{12}:={\rm Id}+u_{12}$.
By the above relation, we have the cocycle condition $\phi_{13}=\phi_{12}\circ \phi_{23}$, so in particular $\phi_{12}$ is an isomorphism of ${\cal O}_{X}$-modules.
Let us give the local description of $\phi_{12}$.
Assume that we are given a logarithmic system of coordinates $\left\{m_{i}| 1\leq i\leq r\right\}$ of $X\to S$.
Then $\left\{\pi^{*}m_{i}| 1\leq i\leq r\right\}$ forms a logarithmic system of coordinates of $X'\to S$.
Let take a lifting $\tilde{m'}_{i}$ of $\pi^{*}m_{i}$.
Then, since $\tilde{F}_{1}$ (resp. $\tilde{F}_{2}$) is a log strong lifting,
we can write $\tilde{F}_{1}^{*}(\tilde{m'}_{i})=\tilde{m}_{i}^{p^{m+1}}(1+p\tilde{g}_{i}^{p^{m}})$ (resp. $\tilde{F}_{2}^{*}(\tilde{m'}_{i})=\tilde{m}_{i}^{p^{m+1}}(1+p\tilde{h}_{i}^{p^{m}})$) for some $\tilde{g}_{i}\in {\cal O}_{\tilde{X}}$ (resp. $\tilde{h}_{i}\in {\cal O}_{\tilde{X}}$)
for each $i$.
By construction of $\phi_{12}$, we then have the following lemma.
\begin{lemm}\label{l9}
\begin{equation*}
\phi_{12}\left(F_{X/S}^{*}\left(\log \pi^{*}m_{i} \right)\right)= F_{X/S}^{*}\left(\log \pi^{*}m_{i}\right)+ \left(g_{i}-h_{i} \right)^{p^{m}}.
\end{equation*}
\end{lemm}
We need the following lemma.
\begin{lemm}
$\phi_{12}$ is an isomorphism of ${\cal D}_{X/S}^{(m)}$-modules.
\end{lemm}
\begin{proof}
It suffices to show that $\phi_{12}$ is compatible with $\epsilon_{n}$.
We may work \'etale locally on $X$.
Let $\left\{m_{i}| 1\leq i\leq r\right\}$ a logarithmic system of coordinates  of $X\to S$.
Let $g_{i}$ and $h_{i}$ be as in Lemma \ref{l9}.
$\epsilon'_{n, \tilde{F_{1}}}\circ \phi_{12}(1)=\phi_{12}\otimes {\rm id}\circ \epsilon'_{n, \tilde{F_{2}}}(1)$ is obvious.
By Lemma \ref{l9} and Lemma \ref{l2}, $\epsilon'_{n, \tilde{F_{1}}}\circ \phi_{12}\left(F_{X/S}^{*}\left(d{\log}\pi^{*}m_{i}\right)\right)$ is equal to 
\begin{eqnarray*}
&& F_{X/S}^{*}\left(d{\log}\pi^{*}m_{i}\right)\otimes 1+ 1\otimes \Psi_{n}(F^{*}_{X/S}d{\log}\pi^{*}m_{i}) +1\otimes (g_{i}-h_{i})^{p^{m}}\\ 
&=&F_{X/S}^{*}\left(d{\log}\pi^{*}m_{i}\right)\otimes 1- 1\otimes \left(\eta_{m_{i}}^{\{p^{m+1}\}} + \sum_{k=1}^{p-1}\frac{(-1)^{p^{m}k}}{k}\eta_{m_{i}}^{p^{m}k}\right)-\left(1\otimes g_{i}-g_{i}\otimes1\right)^{p^{m}}+1\otimes (g_{i}-h_{i})^{p^{m}}\\
&=&F_{X/S}^{*}\left(d{\log}\pi^{*}m_{i}\right)\otimes 1- 1\otimes \left(\eta_{m_{i}}^{\{p^{m+1}\}} + \sum_{k=1}^{p-1}\frac{(-1)^{p^{m}k}}{k}\eta_{m_{i}}^{p^{m}k}\right)-\left(1\otimes h_{i}-g_{i}\otimes1\right)^{p^{m}}
\end{eqnarray*}
On the other hand, by Lemma \ref{l9} and Lemma \ref{l2}, $\phi_{12}\otimes {\rm id}\circ \epsilon'_{n, \tilde{F_{2}}}(F_{X/S}^{*}d{\log}\pi^{*}m_{i})$ is equal to 
\begin{eqnarray*}
&&\phi_{12}\otimes {\rm id}\left( F_{X/S}^{*}\left(d{\log}\pi^{*}m_{i}\right)\otimes1 - 1\otimes \left(\eta_{m_{i}}^{\{p^{m+1}\}} + \sum_{k=1}^{p-1}\frac{(-1)^{p^{m}k}}{k}\eta_{m_{i}}^{p^{m}k}\right)-\left(1\otimes h_{i}-h_{i}\otimes1\right)^{p^{m}} \right)\\
&=& F_{X/S}^{*}\left(d{\log}\pi^{*}m_{i}\right)\otimes 1 +\left(g_{i}-h_{i} \right)^{p^{m}} \otimes 1- 1\otimes \left(\eta_{m_{i}}^{\{p^{m+1}\}} + \sum_{k=1}^{p-1}\frac{(-1)^{p^{m}k}}{k}\eta_{m_{i}}^{p^{m}k}\right)-\left(1\otimes h_{i}-h_{i}\otimes1\right)^{p^{m}}\\
&=&F_{X/S}^{*}\left(d{\log}\pi^{*}m_{i}\right)\otimes 1 - 1\otimes \left(\eta_{m_{i}}^{\{p^{m+1}\}} + \sum_{k=1}^{p-1}\frac{(-1)^{p^{m}k}}{k}\eta_{m_{i}}^{p^{m}k}\right)-\left(1\otimes h_{i}-g_{i}\otimes1\right)^{p^{m}}.
\end{eqnarray*}
This finishes the proof.
\end{proof}
We now construct a ${\cal D}_{X/S}^{(m)}$-module ${\cal E}'$ by a glueing argument and show that ${\cal E}'$ is isomorphic to ${\cal E}_{{\cal X/S}}^{(m)}$.
By Remark \ref{r2}, there exists an \'etale covering $\left\{U_{i}\right\}$ of $X$ with log strong liftings $\tilde{F}_{i}$ of $F_{U_{i}/S}$.
By the above argument, we obtain a ${\cal D}_{U_{i}/S}^{(m)}$-module ${\cal O}_{U_{i}}\oplus F_{U_{i}/S}^{*}{\Omega_{U_{i}'/S}^{1}}$ for sech $i$ and an isomorphism of ${\cal D}_{U_{ij}/S}^{(m)}$-modules $\phi_{ij}: \left({\cal O}_{U_{i}}\oplus F_{U_{i}/S}^{*}{\Omega_{U_{i}'/S}^{1}}\right)|_{U_{ij}}\xrightarrow{\cong} \left({\cal O}_{U_{j}}\oplus F_{U_{j}/S}^{*}{\Omega_{U_{j}'/S}^{1}}\right)|_{U_{ij}}$, where $U_{ij}$ is $U_{i}\times_{X}U_{j}$.
Then $\left\{{\cal O}_{U_{i}}\oplus F_{U_{i}/S}^{*}{\Omega_{U_{i}'/S}^{1}}, \phi_{ij}\right\}$ defines a glueing data of ${\cal D}_{X/S}^{(m)}$-modules and it descents to  a ${\cal D}_{X/S}^{(m)}$-module ${\cal E}'$.
We prove the following theorem.
\begin{theo}\label{t21}
${\cal E}'$ is isomorphic to ${\cal E}_{{\cal X/S}}^{(m)}$ as ${\cal D}_{X/S}^{(m)}$-modules.
\end{theo}
\begin{proof}
Let us remember the locally split exact sequence (\ref{e9}).
If we are given a log strong lifting $\tilde{F}_{1}$, then we have a section $\sigma_{\tilde{F}_{1}}: F^{*}_{X/S}{\Omega}^{1}_{X'/S}\to {\cal E}_{{\cal X/S}}^{(m)}; \omega\mapsto [f\mapsto \langle \omega,f-\tilde{F}_{1}\rangle]$ and we have
\begin{equation*}
\phi'_{\tilde{F}_{1}}: \left({\cal O}_{X}\oplus F_{X/S}^{*}{\Omega_{X'/S}^{1}}\right)\xrightarrow{\cong} {\cal E}_{{\cal X/S}}^{(m)}; (a, \omega)\mapsto a+ \sigma_{\tilde{F}_{1}}(\omega).
\end{equation*}
We have to see the agreement of the glueing datum.
Let $\tilde{F}_{2}$ be another log strong lifting with the same domain of $\tilde{F}_{1}$.
We have $\phi'_{\tilde{F}_{2}}: \left({\cal O}_{X}\oplus F_{X/S}^{*}{\Omega_{X'/S}^{1}}\right)\xrightarrow{\cong} {\cal E}_{{\cal X/S}}^{(m)}$ in a similar manner.
We then define
\begin{equation*}
\phi'_{12}:={\phi'}_{\tilde{F}_{1}}^{-1}\circ\phi'_{\tilde{F}_{2}}: \left({\cal O}_{X}\oplus F_{X/S}^{*}{\Omega_{X'/S}^{1}}\right)\xrightarrow{\cong} {\cal E}_{{\cal X/S}}^{(m)}\xleftarrow{\cong}\left({\cal O}_{X}\oplus F_{X/S}^{*}{\Omega_{X'/S}^{1}}\right).
\end{equation*}
Let us check $\phi_{12}=\phi'_{12}$.
Since the question is \'etale local on $X$, we may work with a logarithmic system of coordinates $\left\{m_{i}| 1\leq i\leq r\right\}$ of $X\to S$.
Let $g_{i}$ and $h_{i}$ be as in Lemma \ref{l9}.
We have
\begin{eqnarray*}
\phi'_{12}(0, F_{X/S}^{*}d{\log}\pi^{*}m_{i})&=&(\sigma_{\tilde{F}_{2}}(F_{X/S}^{*}d{\log}\pi^{*}m_{i})-\sigma_{\tilde{F}_{1}}(F_{X/S}^{*}d{\log}\pi^{*}m_{i}), F_{X/S}^{*}d{\log}\pi^{*}m_{i}).
\end{eqnarray*}
As a section of ${\cal E}_{{\cal X/S}}^{(m)}$, $\sigma_{\tilde{F}_{2}}(F_{X/S}^{*}d{\log}\pi^{*}m_{i})-\sigma_{\tilde{F}_{1}}(F_{X/S}^{*}d{\log}\pi^{*}m_{i})$ sends $f\in {\cal L}^{(m)}_{\cal X/S}$ to $\langle F_{X/S}^{*}d{\log}\pi^{*}m_{i}, \tilde{F}_{1}-\tilde{F}_{2}\rangle$.
Here $\tilde{F}_{1}-\tilde{F}_{2}$ is a section of $F_{X/S}^{*}{\cal T}_{X'/S}$ characterized by 
\begin{equation*}
F^{*}_{X/S}(d\log \pi^{*}m_{i})\mapsto f \text{ where }\tilde{F}_{2}^{*}(\tilde{m_{i}})(1+p\tilde{f})=\tilde{F}_{1}^{*}(\tilde{m_{i}}).
\end{equation*}
So $\sigma_{\tilde{F}_{2}}(F_{X/S}^{*}d{\log}\pi^{*}m_{i})-\sigma_{\tilde{F}_{1}}(F_{X/S}^{*}d{\log}\pi^{*}m_{i})$ is a constant function on ${\cal L}^{(m)}_{\cal X/S}$, which is nothing but $f=\left(g_{i}-h_{i}\right)^{p^{m}}$.
This shows $\phi_{12}=\phi'_{12}$ by Lemma \ref{l9}.
Finally, let us prove the agreement of the ${\cal D}_{X/S}^{(m)}$-actions.
To do this, we explicitly calculate the log $m$-stratification on ${\cal E}_{{\cal X/S}}^{(m)}$ under the existence of log strong lifting $\tilde{F_{1}}$ and a logarithmic system of coordinates $\left\{m_{i}| 1\leq i\leq r\right\}$ of $X\to S$ (cf, the proof of [O, Lemma 5.17]).
Let $\tilde{X}\hookrightarrow \tilde{P}$ be the log $m$-PD envelope of the diagonal of $\tilde{X}$ over $\tilde{S}$.
Let $P$ denote the modulo $p$ reduction of $\tilde{P}$ and $p_{0}$ and $p_{1}$ denote the first and the second projection of $\tilde{P}$ respectively.
By definition, $\epsilon'_{n, \tilde{F_{1}}}(F_{X/S}^{*}d{\log}\pi^{*}m_{i})=F_{X/S}^{*}\left(d{\log}\pi^{*}m_{i}\right)\otimes 1- 1\otimes \left(\eta_{m_{i}}^{\{p^{m+1}\}} + \sum_{k=1}^{p-1}\frac{(-1)^{p^{m}k}}{k}\eta_{m_{i}}^{p^{m}k}\right)-\left(1\otimes g_{i}-g_{i}\otimes1\right)^{p^{m}}$.
We have to prove that the image of $F_{X/S}^{*}d{\log}\pi^{*}m_{i}$ under the map defined by
\begin{equation*}
F^{*}_{X/S}\Omega^{1}_{X'/S}\xrightarrow{\sigma_{\tilde{F_{1}}}} {\cal E}_{\cal X/S}^{(m)} \xrightarrow{p^{*}_{1}} {\cal O}_{P}\otimes {\cal E}_{\cal X/S}^{(m)}\xrightarrow{\cong} {\cal E}_{\cal X/S}^{(m)}\otimes {\cal O}_{P}\xrightarrow {\cong}\left({\cal O}_{X}\oplus F_{X/S}^{*}{\Omega_{X'/S}^{1}}\right)\otimes_{{\cal O}_{X}}{\cal O}_{P}
\end{equation*}
is equal to $\epsilon'_{n, \tilde{F_{1}}}(F_{X/S}^{*}d{\log}\pi^{*}m_{i})$. 
Here the first isomorphism is the HPD-stratification associated to the log $m$-crystal structure on ${\cal E}_{\cal X/S}^{(m)}$ and the second one is 
${\phi'}^{-1}_{\tilde{F}_{1}}\otimes {\rm id}_{{\cal O}_{P}}$.
Let $f_{P/S}$ be the composition $P\to X\xrightarrow{F_{X/S}}X'$.
Let ${\cal L}$ be an \'etale sheaf on $\tilde{P}$ defined by 
\begin{equation*}
\text{for each \'etale open $\tilde{U}$ of $\tilde{P}$}\mapsto \text{\{$f: \tilde{U}\to \tilde{X'}$ such that $f$ mod $p=f_{P/S}|_{U}$\}},
\end{equation*}
where $U$ denotes the modulo $p$ reduction of $\tilde{U}$.
Let us calculate the image of $\sigma_{\tilde{F}_{1}}(F_{X/S}^{*}d{\log}\pi^{*}m_{i})$ under the composition
\begin{equation}\label{e19}
{\cal E}_{{\cal X/S}}^{(m)}\xrightarrow{p_{1}^{*}} {\cal O}_{P}\otimes_{{\cal O}_{X}}{\cal E}_{{\cal X/S}}^{(m)} \xrightarrow{\cong} {\cal E}_{{\cal X/S}}^{(m)}\otimes_{{\cal O}_{X}}{\cal O}_{P}.
\end{equation}
We denote this image by $B$.
As a map $B: {\cal L}\to {\cal O}_{P}$, $B$ sends $\tilde{F'}\in {\cal L}$ to $b$ satisfying
\begin{equation*}
\left(\tilde{F_{1}}\circ p_{1}\right)^{*}\left(\tilde{m_{i}'}\right)(1+p\tilde{b})=\left(\tilde{F'}\circ p_{0}\right)^{*}\left(\tilde{m_{i}'}\right),
\end{equation*}
where $\tilde{b}\in {\cal O}_{\tilde{P}}$ is a lift of $b$.
Let us define the map $C: {\cal L}\to {\cal O}_{P}$ by $\tilde{F'}\in {\cal L}\mapsto c$. 
Here $c$ is characterized by 
\begin{equation*}
\left(\tilde{F_{1}}\circ p_{1}\right)^{*}\left(\tilde{m_{i}'}\right)(1+p\tilde{c})=\left(\tilde{F_{1}}\circ p_{0}\right)^{*}\left(\tilde{m_{i}'}\right),
\end{equation*}
where $\tilde{c}\in {\cal O}_{\tilde{P}}$ is a lift of $c$.
We also define the map $D: {\cal L}\to {\cal O}_{P}$ by $\tilde{F'}\in {\cal L}\mapsto d$. 
Here $d$ is characterized by 
\begin{equation*}
\left(\tilde{F_{1}}\circ p_{0}\right)^{*}\left(\tilde{m_{i}'}\right)(1+p\tilde{d})=\left(\tilde{F'}\circ p_{0}\right)^{*}\left(\tilde{m_{i}'}\right),
\end{equation*}
where $\tilde{d}\in {\cal O}_{\tilde{P}}$ is a lift of $d$.
Then $C$ and $D$ can be regarded as a section of ${\cal E}_{{\cal X/S}}^{(m)}\otimes_{{\cal O}_{X}}{\cal O}_{P}$ and $B=C+D$.
Now, by the similar calculation in the proof of Lemma \ref{l2} (note that $p!=-p$ modulo $p^{2}$), we have 
\begin{equation*}
c=-\eta_{m_{i}}^{\{p^{m+1}\}} - \sum_{k=1}^{p-1}\frac{(-1)^{p^{m}k}}{k}\eta_{m_{i}}^{p^{m}k}-\left(1\otimes g_{i}-g_{i}\otimes1\right)^{p^{m}}.
\end{equation*}
Therefore $C$ is a constant function and ${\phi'}_{\tilde{F}_{1}}^{-1}\otimes {\rm id}_{{\cal O}_{P}}(C)=-\eta_{m_{i}}^{\{p^{m+1}\}} - \sum_{k=1}^{p-1}\frac{(-1)^{p^{m}k}}{k}\eta_{m_{i}}^{p^{m}k}-\left(1\otimes g_{i}-g_{i}\otimes1\right)^{p^{m}}$.
On the other hand, by definition, $d$ can be described by the pullback of $d'\in {\cal O}_{X}$ via $p_{0}$, where $d'$ is characterized by its lift $\tilde{d'}\in {\cal O}_{\tilde{X}}$ satisfying
\begin{equation*}
\tilde{F_{1}}^{*}\left(\tilde{m_{i}'}\right)(1+p\tilde{d'})=\tilde{F'}^{*}\left(\tilde{m_{i}'}\right).
\end{equation*}
By definition, $\sigma_{\tilde{F}_{1}}(F_{X/S}^{*}d{\log}\pi^{*}m_{i})$ is the map ${\cal L}_{\cal X/S}^{(m)}\to {\cal O}_{X}$ defined by $\tilde{F'}\mapsto d'$. 
Thus $D$ corresponds to $F_{X/S}^{*}d{\log}\pi^{*}m_{i}\otimes 1$ via the isomorphism ${\phi'}_{\tilde{F}_{1}}^{-1}\otimes {\rm id}: {\cal E}_{{\cal X/S}}^{(m)}\otimes_{{\cal O}_{X}}{\cal O}_{P}\xrightarrow {\cong}\left({\cal O}_{X}\oplus F_{X/S}^{*}{\Omega_{X'/S}^{1}}\right)\otimes_{{\cal O}_{X}}{\cal O}_{P}$.
Therefore $B=C+D$ corresponds to $\epsilon'_{n, \tilde{F_{1}}}(F_{X/S}^{*}d{\log}\pi^{*}m_{i})$ via the isomorphism ${\phi'}_{\tilde{F}_{1}}^{-1}\otimes {\rm id}$.
This finishes the proof.
\end{proof}

As a consequence of Theorem \ref{t21}, the global version of a splitting module $\check{{\cal K}}^{(m), {\cal A}}_{\cal X/S}$ is reconstructed by glueing based on the maps $\Psi_{n}$ and $\Psi$ defined in a local situation.
We also know that the log global Cartier transform of higher level is a glueing of the local Cartier transform of higher level.

\section*{Acknowledgments} 
The author expresses his hearty thanks to his supervisor Atsushi Shiho for many useful discussions and the careful reading of this paper.
He would like to thank to Nobuo Tsuzuki, Atsushi Shiho and Tomoyuki Abe
for inviting him to give a talk at a conference held at Tohoku University, and to Ahmed Abbes and Fabrice Orgogozo for inviting him to give a talk at a conference held at IH\'ES.
This paper owes its existence to the work of Gros-Le Stum-Quir\'os \cite{GLQ},
Lorenzon \cite{L}, Montagnon \cite{M}, Ogus-Vologodsky \cite{OV}  and Schepler \cite{S}.
He thanks them heartily. 
This work was supported by the Program for Leading Graduate 
Schools, MEXT, Japan and the Grant-in-Aid for JSPS fellows.

\end{document}